\pdfoutput=1
\documentclass[a4paper,11pt,reqno]{amsart}%
\usepackage{amsfonts}
\usepackage{amsthm}
\usepackage{hyperref}
\usepackage{amssymb}
\usepackage{graphicx}
\usepackage{float}
\usepackage[space]{cite}
\usepackage{eqparbox,array}
\usepackage{algorithm}
\usepackage[noend]{algorithmic}
\usepackage{dcolumn}
\usepackage{amsmath}
\usepackage{verbatim}
\usepackage{todonotes}%
\newcolumntype{r}{D{.}{.}{-1}}

\newtheorem{theorem}{Theorem}[section]
\theoremstyle{plain}

\newtheorem{corollary}[theorem]{Corollary}
\theoremstyle{definition}
\newtheorem{definition}[theorem]{Definition}
\newtheorem{example}[theorem]{Example}
\theoremstyle{plain}
\newtheorem{lemma}[theorem]{Lemma}
\newtheorem{proposition}[theorem]{Proposition}
\theoremstyle{remark}
\newtheorem{remark}[theorem]{Remark}
\newtheorem{notation}[theorem]{Notation}
\numberwithin{equation}{section}

\DeclareMathOperator{\Hom}{Hom}

\DeclareMathOperator{\QQ}{\mathbb Q}

\begin{document}
\title[Algorithms for the adjoint ideal of a curve]{Local to global algorithms for the Gorenstein adjoint ideal of a curve}
\author[J. B\"{o}hm]{Janko B\"{o}hm}
\address{Fachbereich Mathematik\\
Universit\"{a}t Kaiserslautern\\
Postfach 3049\\
D-67653 Kaiserslautern, Germany}
\email{boehm@mathematik.uni-kl.de}
\urladdr{\href{http://www.mathematik.uni-kl.de/~boehm/index.htm}{http://www.mathematik.uni-kl.de/~boehm/index.htm}%
}
\author[W. Decker]{Wolfram Decker}
\address{Fachbereich Mathematik\\
Universit\"{a}t Kaiserslautern\\
Postfach 3049\\
D-67653 Kaiserslautern, Germany}
\email{decker@mathematik.uni-kl.de}
\urladdr{\href{http://www.mathematik.uni-kl.de/~decker/}{http://www.mathematik.uni-kl.de/~decker/}%
}
\author[S. Laplagne]{Santiago Laplagne}
\address{Departamento de Matem\'{a}tica\\
Facultad de Ciencias Exactas y Naturales\\
(1428) Pabell\'{o}n I - Ciudad Universitaria\\
Buenos Aires, Argentina}
\email{slaplagn@dm.uba.ar}
\urladdr{\href{http://cms.dm.uba.ar/Members/slaplagn}{http://cms.dm.uba.ar/Members/slaplagn}%
}
\author[G. Pfister]{Gerhard Pfister}
\address{Fachbereich Mathematik\\
Universit\"{a}t Kaiserslautern\\
Postfach 3049\\
D-67653 Kaiserslautern, Germany}
\email{pfister@mathematik.uni-kl.de}
\urladdr{\href{http://www.mathematik.uni-kl.de/~pfister/}{http://www.mathematik.uni-kl.de/~pfister/}%
}

\begin{abstract}
We present new algorithms for computing adjoint ideals of curves and thus, in
the planar case, adjoint curves. With regard to terminology, we follow
Gorenstein who states the adjoint condition in terms of conductors.

Our main algorithm yields the Gorenstein adjoint ideal $\mathfrak{G}$ of a
given curve as the intersection of what we call local Gorenstein adjoint
ideals. Since the respective local computations do not depend on each other,
our approach is inherently parallel.

Over the rationals, further parallelization is achieved by a modular version
of the algorithm which first computes a number of the characteristic $p$
counterparts of $\mathfrak{G}$ and then lifts these to characteristic zero. As
a key ingredient, we establish an efficient criterion to verify the
correctness of the lift.

Well-known applications are the computation of Riemann-Roch spaces, the
construction of points in moduli spaces, and the parametrization of rational curves.

We have implemented different variants of our algorithms together with Mnuk's
approach \cite{Mnuk} in the computer algebra system \textsc{{Singular}} and
give timings to compare the performance of the algorithms.





\end{abstract}
\maketitle

\section{Introduction}

In classical algebraic geometry, starting from Riemann's paper on abelian
functions \cite{Riemann}, the adjoint curves of an irreducible plane curve
$\Gamma$ have been used as an essential tool in the study of the geometry of
$\Gamma$. The defining property of an \emph{adjoint curve} is that it passes
with \textquotedblleft sufficiently high\textquotedblright\ multiplicity
through the singularities of $\Gamma$. There are several ways of making this
precise, developed in classical papers by \cite{BrillNoether}, \cite[1893]%
{Castelnuovo}, and \cite{Petri}, and in more recent work by \cite{Groebner41,
Gorenstein} and \cite{vanderWaerden, Keller}. We refer to \cite{Keller65}%
,\cite{GrecoValabrega}, \cite{GrecoValabrega2}, and \cite{Ciliberto} for
results comparing the different notions: whereas the adjoint condition given
by Brill and Noether is more restrictive, the notions of adjoint curves given
by the other authors above coincide.

In this paper, we always consider adjoint curves in the less restrictive
sense. In fact, we rely on Gorenstein's algebraic definition which states the
adjoint condition at a singular point $P\in\Gamma$ by considering the
conductor of the local ring $\mathcal{O}_{\Gamma,P}$ in its normalization. It
is a well-known consequence of Max Noether's Fundamentalsatz that the adjoint
curves of any given degree $m$ cut out, residual to a fixed divisor supported
on the singular locus of $\Gamma$, a complete linear series. Of fundamental
importance is the case $m=\deg\Gamma-3$ which, as shown by Gorenstein, yields
the canonical series.

The ideal generated by the defining forms of the adjoint curves of $\Gamma$ is
called the \emph{adjoint ideal} of $\Gamma$. In \cite{ArbarelloCiliberto}, the
concept of adjoint ideals is extended to the non-planar case: consider a
non-degenerate irreducible curve $\Gamma\subset\mathbb{P}_{k}^{r}%
=\operatorname*{Proj}(S)$, and let $I$ be a saturated homogeneous ideal of $S$
which is supported on the singular locus of $\Gamma$. Then, roughly speaking,
$I$ is an adjoint ideal of $\Gamma$ if its homogeneous elements of degree
$m\gg0$ cut out, residual to a fixed divisor supported on the singular locus,
a complete linear series. As pointed out in \cite{ArbarelloCiliberto}, the
existence of adjoint ideals is implicit in classical papers: examples are the
Castelnuovo adjoint ideal and the Petri adjoint ideal. In \cite{Ciliberto}, it
is shown that Gorenstein's condition leads to the largest possible adjoint
ideal, containing all other adjoint ideals, and now referred to as the
\emph{Gorenstein adjoint ideal }$\mathfrak{G}=\mathfrak{G}(\Gamma)$. See
\cite{Ciliberto} for some remarks on how the different concepts of adjoint
ideals compare in the non-planar case.


With regard to practical applications, adjoint curves enter center stage in
the classical Brill-Noether algorithm for computing Riemann-Roch spaces, which
in turn can be used to construct Goppa codes (see \cite{BrRi}). Furthermore,
linear series cut out by adjoint curves allow us to construct explicit
examples of smooth curves via singular plane models; a typical application is
the experimental study of moduli spaces of curves. If the geometric genus of a
plane curve $\Gamma$ is zero, then the adjoint curves of degree $\deg\Gamma-2$
specify a birational map to a rational normal curve. Based on this, we can
find an explicit parametrization of $\Gamma$ over its field of definition,
starting either from the projective line or a conic. See \cite{Boehm, BDLSpar}
and the implementation in the \textsc{Singular} library \cite{BDLS3}.
Algorithms for parametrization, in turn, have applications in computer aided
design, for example, to compute intersections of curves with other algebraic
varieties. See also \cite{Winkler}.

A well-known algorithm for computing the Gorenstein adjoint ideal
$\mathfrak{G} =\mathfrak{G}(\Gamma)$ in the planar case is due to \cite{Mnuk}.
This algorithm makes use of linear algebra to obtain $\mathfrak{G}$ from an
integral basis for the normalization $\overline{k[C]}$, where $C$ is an affine
part of $\Gamma$ containing all singularities of $\Gamma$. Efficient ways of
finding integral bases rely on Puiseux series techniques (see
\cite{vanHoeijIntegralBasis}, \cite{{BDLS}}). This somewhat limits Mnuk's
approach to characteristic zero. The same applies to the algorithm of
\cite{KM}, which also computes the Gorenstein adjoint ideal of a plane curve
from an integral basis of $\overline{k[C]}$. The approach of \cite{OR14}, on
the other hand, is limited to ordinary multiple points.

In this paper, we present a new algorithm for computing $\mathfrak{G}$. This
algorithm is highly efficient and not restricted to the planar case, special
types of singularities or to characteristic zero. The basic idea is to compute
$\mathfrak{G}$ as the intersection of \textquotedblleft local Gorenstein
ideals\textquotedblright, one for each singular point of $\Gamma$. Each local
ideal is obtained via Gr\"obner bases, starting from a \textquotedblleft local
contribution\textquotedblright\ to the normalization $\overline{k[C]}$ at the
respective singular point. To find these contributions, we use the algorithm
from \cite{BDLPSS} which is a local variant of the normalization algorithm
designed in \cite{GLS}. In practical terms, given any field of definition
$L\subset k$, we treat the points in a complete set of conjugate singularities simultaneously.

Our approach is already faster per se. In addition, it can take advantage of
handling special classes of singularities in an ad hoc way. Above all, it is
inherently parallel. For input over the rationals, further parallelization is
achieved by a modular version of our algorithm which first computes a number
of characteristic $p$ counterparts of $\mathfrak{G}$ and then lifts these to
characteristic zero. To apply the general rational reconstruction scheme from
\cite{FareyPaper}, we prove an efficient criterion to verify the correctness
of the lift.

Our paper is organized as follows: In Section \ref{sec normalization}, we
discuss algorithmic normalization. In Section \ref{Sec adjoint ideals}, we
review the definition of adjoint ideals and some related facts. In Section
\ref{Sec Global approach}, we describe global algorithmic approaches to obtain
$\mathfrak{G}$. We first discuss Mnuk's approach. Then we describe a global
approach which relies on normalization and Gr\"{o}bner bases.
In Sections \ref{Sec Global from local} and \ref{Sec General local approch},
we present our local to global algorithm for finding $\mathfrak{G}$ via
normalization and Gr\"{o}bner bases.
Section \ref{sec summary} pays particular attention to the planar case,
commenting on the direct treatment of special types of singularities. In
Section \ref{sec parallel and modular}, we discuss the modular version of our
algorithm. Finally, in Section \ref{Sec examples and comparisons}, we compare
the performance of the different approaches, relying on our implementations in
the computer algebra system \textsc{{Singular}}, and running various examples
coming from algebraic geometry.

\section{Algorithms for Normalization\label{sec normalization}}

We begin with some general remarks on normalization and the role played by the
conductor. For these, let $A$ be any reduced Noetherian ring, and let
$\operatorname{Q}(A)$ be its total ring of fractions. Then $\operatorname{Q}%
(A)$ is again a reduced Noetherian ring. We write
\[
\operatorname{Spec}(A)=\{P\subset A\mid P{\text{ prime ideal}}\}
\]
for the {\emph{spectrum}} of $A$. The {\emph{vanishing locus}} of an ideal $J$
of $A$ is the set $V(J)=\{P\in\operatorname{Spec}(A)\mid P\supset J\}$. 

The {\emph{normalization}} of $A$, written $\overline{A}$, is the integral
closure of $A$ in $\operatorname{Q}(A)$. We call $A$
\emph{normalization-finite} if $\overline{A}$ is a finite $A$-module, and we
call $A$ \emph{normal} if $A=\overline{A}$.

We denote by
\[
N(A)=\{P\in\operatorname{Spec}(A)\mid A_{P}\text{ is not normal}\}
\]
the {\emph{non-normal locus}} of $A$, and by
\[
\operatorname{Sing}(A)=\{P\in\operatorname{Spec}(A)\mid A_{P}\text{ is not
regular}\}
\]
the {\emph{singular locus}} of $A$.

\begin{remark}
\label{rem:nnlocus-singlocus-dim-one} Note that $N(A)\subset
\operatorname{Sing}(A)$. Equality holds if $A$ is of pure dimension one.
Indeed, a Noetherian local ring of dimension one is normal iff it is regular
(see \cite[Thm.~4.4.9]{DeJong}).
\end{remark}

\begin{definition}
If $R\subset S$ is an extension of rings, the \emph{conductor} of $A$ in $B$
is
\[
\mathcal{C}_{S/R}=\left\{  r\in R\mid rS\subset S\right\}  .
\]

\end{definition}

Note that $\mathcal{C}_{S/R}$ is the largest ideal of $R$ which is also an
ideal of $S$.

\begin{notation}
If $A$ is a reduced Noetherian ring as above, we write
\[
\mathcal{C}_{A}=\mathcal{C}_{\overline{A}/A}=\{a\in A\mid a\overline{A}\subset
A\}\text{.}%
\]
\end{notation}

\begin{lemma}
\label{lemma-role-of-cond} \label{lemma:role-of-cond} We have $N(A)\subset
V(\mathcal{C}_{A})$. Furthermore, ${A}$ is normalization-finite iff
$\mathcal{C}_{A}$ contains a nonzerodivisor of $A$. In this case,
$N(A)=V(\mathcal{C}_{A})$.
\end{lemma}

\begin{proof}
See \cite[Lemmas~3.6.1, 3.6.3]{GP}.
\end{proof}

\begin{remark}
[Splitting of Normalization]Finding the normalization can be reduced to the
case of integral domains: If $P_{1}\dots, P_{s}$ are the minimal primes of
$A$, then
\[
\overline{A}\cong\overline{A/P_{1}} \times\cdots\times\overline{A/P_{s}}
\]
(see \cite[Thm.~1.5.20]{DeJong}).
\end{remark}

\begin{remark}
Let $k$ be a field. An {\emph{affine $k$-domain}} is a finitely generated
$k$-algebra which is an integral domain. By Emmy Noether's finiteness theorem
(see \cite[Cor.~13.13]{Eis}), any such domain is normalization-finite, and its
normalization is an affine $k$-domain as well. Geometrically, 
by gluing, this implies that any integral variety $X$ over $k$ admits 
a (unique) \emph{normalization map} $\overline{X} \rightarrow X$, where 
$\overline{X}$ is again an integral variety over $k$
(see, for example, \cite[Sec. 4.1.2]{Liu}).
Specifically, by Remark \ref{rem:nnlocus-singlocus-dim-one}, if
$\Gamma$ is a curve over $k$, we get the \emph{nonsingular model}
$\pi:\overline{\Gamma}\rightarrow\Gamma$.
%
\end{remark}

Now, we briefly discuss algorithmic normalization. We begin by recalling the
normalization algorithm of Greuel, Laplagne, and Seelisch \cite{GLS}, which is
an improvement of de Jong's algorithm (see \cite{deJong98}, \cite{DGPJ}). This
algorithm, to which we refer as the GLS Algorithm, is based on the normality
criterion of Grauert and Remmert. To state this criterion, we need:

\begin{lemma}
\label{lemma:prep-GR} Let $A$ be a reduced Noetherian ring, and let $J\subset
A$ be an ideal which contains a nonzerodivisor $g$ of $A$. Then:

\begin{enumerate}
\item If $\varphi\in\Hom_{A}(J, J)$, the fraction
$\varphi(g)/g\in\overline{A}$ is independent of the choice of $g$, and
$\varphi$ is multiplication by $\varphi(g)/g$.

\item There are natural inclusions of rings
\[
A\subset\Hom_{A}(J, J)\cong\frac{1}{g}(gJ :_{A} J)\subset\overline{A}%
\subset{\text{Q}}(A),\; a \mapsto\varphi_{a}, \; \varphi\mapsto\frac
{\varphi(g)}{g},
\]
where $\varphi_{a}$ is multiplication by $a$.
\end{enumerate}
\end{lemma}

\begin{proof}
See \cite[Lemmas~3.6.1, 3.6.3]{GP}.
\end{proof}

\begin{proposition}
[Grauert and Remmert Criterion]\label{prop:testideal}\label{prop:crit-GR} Let
$A$ be a reduced Noetherian ring, and let $J \subset A$ be a radical ideal
which contains a nonzerodivisor $g$ of $A$ and satisfies $V(\mathcal{C}_{A})
\subset V(J)$. Then $A$ is normal iff $A \cong\Hom_{A}(J, J)$ via the map
which sends $a$ to multiplication by $a$.
\end{proposition}

\begin{proof}
See \cite{GR}, \cite[Prop.~3.6.5]{GP}.
\end{proof}

\begin{definition}
\label{def:test-pair} A pair $(J,g)$ as in the proposition is called a
{\emph{test pair}} for $A$, and $J$ is called a {\emph{test ideal}} for $A$.
\end{definition}

If $k$ is a field and $A$ is an affine $k$-domain, then test pairs exist by
Lemma \ref{lemma-role-of-cond} and Emmy Noether's finiteness theorem. If, in
addition, $k$ is perfect, a test pair can be found by applying the Jacobian
criterion (see \cite[Thm. 16.19]{Eis} for this criterion). In fact, in this
case, we may choose the radical of the Jacobian ideal
$M$ together with any nonzero element $g$ of $M$ as a test pair. Given a test
pair $(J,g)$, the basic idea of finding $\overline{A}$ is to enlarge $A$ by a
sequence of finite extensions of affine $k$-domains
\[
A_{i+1}\cong\operatorname*{Hom}\nolimits_{A_{i}}(J_{i},J_{i})\cong\frac{1}%
{g}(gJ_{i}:_{A_{i}}J_{i})\subset\overline{A}\subset\operatorname{Q}(A),
\]
with $A_{0}=A$ and $J_{i}=\sqrt{JA_{i}}$, until the Grauert and Remmert
criterion allows one to stop. According to \cite{GLS}, each $A_{i}$ can be
represented as a quotient $\frac{1}{d_{i}}U_{i}\subset\operatorname{Q}(A)$,
where $U_{i}\subset A$ is an ideal and $d_{i}\in U_{i}$ is nonzero. In this
way, all computations except those of the radicals $J_{i}$ may be carried
through in $A$.

\begin{example}
\label{ex a4 E8} For
\[
A=\mathbb{C}[x,y]=\mathbb{C}[X,Y]/\langle X^{5}-Y^{2}\left(  Y-1\right)
^{3}\rangle\text{,}%
\]
the radical of the Jacobian ideal is
\[
J:=\left\langle x,y\left(  y-1\right)  \right\rangle _{A}\text{,}%
\]
so that we can take $(J,x)$ as a test pair. Then, in its first step, the
normalization algorithm yields
\[
A_{1}=\frac{1}{x}U_{1}=\frac{1}{x}\left\langle x,y(y-1)^{2}\right\rangle _{A}
\text{.}%
\]
In the next steps, we get
\[
A_{2}=\frac{1}{x^{2}}U_{2}=\frac{1}{x^{2}}\left\langle x^{2}%
,xy(y-1),y(y-1)^{2}\right\rangle _{A}%
\]
and%
\[
A_{3}=\frac{1}{x^{3}}U_{3}=\frac{1}{x^{3}}\left\langle x^{3},x^{2}%
y(y-1),xy(y-1)^{2},y^{2}(y-1)^{2}\right\rangle _{A}\text{.}%
\]
In the final step, we find that $A_{3}$ is normal and, hence, equal to
$\overline{A}$.
\end{example}

Next, we describe a local to global variant of the GLS algorithm, given in
\cite{BDLPSS}, which is a considerable enhancement of the algorithm, and which
serves as a motivation for our local to global approach to compute the
Gorenstein adjoint ideal. This variant is based on the following two
observations from \cite{BDLPSS}: First, the normalization $\overline{A}$ can
be computed as the sum of local contributions $A\subset A^{(i)}\subset
\overline{A}$, and second, local contributions can be obtained efficiently by
a local variant of the GLS algorithm. For our purposes here, it is enough to
present the relevant results in a special case. Here, as usual, if $P$ is a
prime of a ring $R$, and $M$ is an $R$-module, we write $M_{P}$ for the
localization of $M$ at $R\setminus P$.

\begin{proposition}
\label{prop:local-to-global-I} Let $A$ be an affine $k$-domain of 
dimension one, and let $\operatorname{Sing}(A)=\{P_{1},\dots,P_{s}\}$ be its
singular locus. For $i=1,\dots,s$, let an intermediate ring $A\subset
A^{(i)}\subset\overline{A}$ be given such that $A_{P_{i}}^{(i)}=\overline
{A_{P_{i}}}$. Then
\[
\sum_{i=1}^{s}A^{(i)}=\overline{A}\text{.}%
\]

\end{proposition}

\begin{proof}
See \cite[Prop. 15]{BDLPSS}.
\end{proof}


\begin{definition}
A ring $A^{(i)}$ as above is called a \emph{local contribution} to
$\overline{A}$ at $P_{i}$. It is called a \emph{minimal local contribution} if
$A^{(i)}_{P_{j}}={A _{P_{j}}}$ for $j\neq i$.

\end{definition}

The computation of local contributions is based on the modified version of the
Grauert and Remmert criterion below:

\begin{proposition}
\label{prop local grauert-remmert} Let $A$ be an affine $k$-domain of
dimension one, let $A\subset A^{\prime}$ be a finite ring extension, let
$P\in\operatorname{Sing}(A)$, and let $J^{\prime}=\sqrt{PA^{\prime}}$. If
\[
A^{\prime}\cong\operatorname*{Hom}\nolimits_{A^{\prime}}(J^{\prime},J^{\prime
})
\]
via the map which sends $a^{\prime}$ to multiplication by $a^{\prime}$, then
$A_{P}^{\prime}$ is normal.
\end{proposition}

\begin{proof}
See \cite[Prop. 16]{BDLPSS}.
\end{proof}

Considering an affine domain $A$ of dimension one over a perfect field
$k$, let $P\in\operatorname{Sing}(A)$. Choose $P$ together with a nonzero
element $g$ in $P$ instead of a test pair as in Definition \ref{def:test-pair}%
. Then, proceeding as before, we get a chain of affine $k$-domains
\[
A\subset A_{1}\subset\dots\subset A_{m}\subset\overline{A}%
\]
such that $A_{m}$ is a local contribution to $\overline{A}$ at $P$.

\begin{remark}
Given $A$ as above, a finite ring extension $A\subset A^{\prime}$, and a prime
$P\in\operatorname{Sing}(A)$, let $Q\in\operatorname{Sing}(A)$ be a prime
different from $P$, and let $J^{\prime}=\sqrt{PA^{\prime}}$. Then
\begin{align*}
\Hom_{A^{\prime}}(J^{\prime},J^{\prime})_{Q}  &  \cong\Hom_{A^{\prime}_{Q}%
}(J^{\prime}_{Q},J^{\prime}_{Q})\\
&  \cong\Hom_{A^{\prime}_{Q}}(A^{\prime}_{Q},A^{\prime}_{Q})\cong A^{\prime
}_{Q}%
\end{align*}
(see \cite[Proposition 2.10]{Eis}). Inductively, this shows that the algorithm
outlined above computes a minimal local contribution to $\overline{A}$ at $P$.
Note that such a contribution is uniquely determined since, by definition, its
localization at each $P\in\operatorname*{Spec}(A)$ is determined.
\end{remark}

\begin{example}
In the case of Example \ref{ex a4 E8}, there are two singularities
$P_{1}=\left\langle x,y\right\rangle $ and $P_{2}=\left\langle
x,y-1\right\rangle $. For $P_{1}$, the local normalization algorithm yields
$\overline{A_{P_{1}}}=(\frac{1}{d_{1}}U_{1})_{P_{1}}$, where
\[
d_{1}=x^{2}\;{\text{ and }}\;U_{1}=\left\langle x^{2},\text{ }y(y-1)^{3}%
\right\rangle _{A}\text{.}%
\]
For $P_{2}$, we get\thinspace\ $\overline{A_{P_{2}}}=(\frac{1}{d_{2}}%
U_{2})_{P_{2}}$, where%
\[
d_{2}=x^{3}\;{\text{ and }}\;U_{2}=\left\langle x^{3},\text{ }x^{2}%
y^{2}\left(  y-1\right)  ,\text{ }y^{2}\left(  y-1\right)  ^{2}\right\rangle
_{A}\text{.}%
\]
Combining the local contributions, we get%
\[
\frac{1}{d}U=\frac{1}{d_{1}}U_{1}+\frac{1}{d_{2}}U_{2}%
\]
with $d=x^{3}$ and
\[
U=\left\langle x^{3},\text{ }xy(y-1)^{3},\text{ }x^{2}y^{2}\left(  y-1\right)
,\text{ }y^{2}\left(  y-1\right)  ^{2}\right\rangle _{A}\text{.}%
\]
Note that $U$ coincides with the ideal $U_{3}$ computed in Example
\ref{ex a4 E8}.
\end{example}

\begin{notation}
In our applications, $A$ will always be the coordinate ring $k[C]=k\left[
X_{1},...,X_{r}\right]  /I(C)$ of an integral affine curve $C\subset
\mathbb{A}_{k}^{r}$ over a perfect field $k$. Given a 
point\footnote{The term \emph{point} will always refer to a closed point.} $P\in C$,
by abuse of notation, if $I\subset k[x_{1}%
,\dots,x_{r}]$ is an ideal properly containing $I(C)$, we will write $I_{P}$
for the ideal of the local ring $\mathcal{O}_{C,P}$ obtained by mapping $I$ to
$k[C]$ and localizing at $P$. Likewise for the homogeneous localization of a
homogeneous ideal in the projective case.
\end{notation}

\section{Adjoint ideals\label{Sec adjoint ideals}}

Let $k$ be a field, and let $\Gamma\subset\mathbb{P}_{k}^{r}$ be an integral 
non-degenerate projective curve. Write
$S=k[X_{0},...,X_{r}]$ for the homogeneous coordinate ring of $\mathbb{P}%
_{k}^{r}$, $I(\Gamma)\subset S$ for the homogeneous ideal of $\Gamma$,
$k[\Gamma]=S/I(\Gamma)$ for the homogeneous coordinate ring of $\Gamma$, and
$\operatorname{Sing}(\Gamma)$ for the singular locus of $\Gamma$.

Let $\pi:\overline{\Gamma}\rightarrow\Gamma$ be the normalization map, let $P$
be a point of $\Gamma$, and let $\mathcal{O}_{\Gamma,P}$ be the local ring of
$\Gamma$ at $P$. Then the normalization $\overline{\mathcal{O}_{\Gamma,P}}$ is
a semi-local ring whose maximal ideals correspond to the points of
$\overline{\Gamma}$ lying over $P$. Furthermore, $\overline{\mathcal{O}%
_{\Gamma,P}}$ is finite over $\mathcal{O}_{\Gamma,P}$ and, thus, a
finite-dimensional $k$-vector space. The dimension
\[
\delta_{P}(\Gamma)=\delta(\mathcal{O}_{\Gamma,P})=\dim_{k}\overline
{\mathcal{O}_{\Gamma,P}}/\mathcal{O}_{\Gamma,P}%
\]
is called the \emph{delta invariant} of $\;\!\Gamma$ at $P$.
The \emph{arithmetic genus} of $\Gamma$ is $p_{a}(\Gamma)=1-\operatorname*{P}%
_{\Gamma}(0)$, where $\operatorname*{P}_{\Gamma}$ is the Hilbert polynomial of
$k[\Gamma]$. Making use of the (global) \emph{delta invariant}
\[
\delta(\Gamma)={\textstyle\sum\nolimits_{P\in\operatorname*{Sing}(\Gamma)}%
}\delta_{P}(\Gamma)\
\]
of $\Gamma$, the \emph{geometric genus} $p(\Gamma)$ of $\Gamma$ is given by%
\[
p(\Gamma)=p(\overline{\Gamma})=p_{a}(\Gamma)-\delta(\Gamma)
\]
(see \cite{Hironaka}). If $\Gamma$ is a plane curve of degree $n$, we have
$p_{a}(\Gamma)=\binom{n-1}{2}$.

Following the presentation in \cite{Chiarli}, we now recall the definition and 
characterization of adjoint ideals due to \cite{ArbarelloCiliberto} and \cite{Ciliberto}.
Let
$I=\bigoplus_{m\geq0}I_{m}\subset S=k[X_{0},...,X_{r}]$ be a saturated
homogeneous ideal properly containing $I(\Gamma)$.
Pulling back $\operatorname*{Proj}(S/I)$ via $\pi$, we get an effective
divisor $\Delta(I)$ on $\overline{\Gamma}$. Let $H$ be a divisor on
$\overline{\Gamma}$ given as the pullback of a hyperplane in
$\mathbb{P}_{k}^{r}$. Then, since any divisor on $\overline{\Gamma}$ cut out
by a homogeneous polynomial in $I$ is of the form $D+\Delta(I)$ for some
effective divisor $D$, we have natural linear maps
\[
\varrho_{m}:I_{m}\rightarrow H^{0}\left(  \overline{\Gamma},\mathcal{O}%
_{\overline{\Gamma}}\left(  mH-\Delta(I)\right)  \right)  ,
\]
for all $m\geq0$.

\begin{remark}
\label{rmk kernel rho}
Consider the exact sequence
\[
0\rightarrow\widetilde{I}\mathcal{O}_{\Gamma}\rightarrow\pi_{\ast}%
(\widetilde{I}\mathcal{O}_{\overline{\Gamma}})\rightarrow\mathcal{F}%
\rightarrow0,
\]
where $\widetilde{I}$ is the ideal sheaf associated to $I$, and $\mathcal{F}$ is the cokernel.
Taking global sections, we get, for $m\gg0$, the exact sequence
\[
0\rightarrow H^{0}\big(  \Gamma,\widetilde{I}\mathcal{O}_{\Gamma}(m)\big)
\rightarrow H^{0}\big(  \overline{\Gamma},\widetilde{I}\mathcal{O}%
_{\overline{\Gamma}}(mH)\big)  \rightarrow H^{0}\left(  \Gamma
,\mathcal{F}\right)  \rightarrow0.
\]
Indeed, $\mathcal{F}$ has finite support and, since the normalization map $\pi$ is finite, we have  $
H^{0}\big(  \overline{\Gamma},\widetilde{I}\mathcal{O}_{\overline{\Gamma}}(mH)\big) \cong 
H^{0} \big( {\Gamma}, \pi_{\ast}(\widetilde{I}\mathcal{O}_{\overline{\Gamma}})(m)\big)$. 
Since $\widetilde{I}\mathcal{O}_{\overline{\Gamma}}(mH)=\mathcal{O}_{\overline{\Gamma}}\big(  mH-\Delta
(I)\big)  $ and, for $m\gg0$, $H^{0}\big(  \Gamma,\widetilde{I}\mathcal{O}_{\Gamma
}(m)\big)  =I_{m}/I(\Gamma)_{m}$, we get, for $m\gg0$, the exact sequence
\[
0\rightarrow I_{m}/I(\Gamma)_{m}\overset{\overline{\varrho_{m}}}{\rightarrow
}H^{0}\left(  \overline{\Gamma},\mathcal{O}_{\overline{\Gamma}}\left(
mH-\Delta(I)\right)  \right)  \rightarrow H^{0}\left(  \Gamma,\mathcal{F}%
\right)  \rightarrow0.
\]
In particular, for $m\gg0$,
\[
\ker(\varrho_{m})=I(\Gamma)_{m}.
\]
\end{remark}

\begin{definition}
With notation and assumptions as above, the ideal $I$ is called an
\emph{adjoint ideal} of $\Gamma$
if the maps%
\[
\varrho_{m}:I_{m}\rightarrow H^{0}\left(  \overline{\Gamma},\mathcal{O}%
_{\overline{\Gamma}}\left(  mH-\Delta(I)\right)  \right)
\]
are surjective for $m$ large enough.
\end{definition}

As already remarked in the introduction, the existence of adjoint ideals is
classical. Locally, adjoint ideals are characterized by the following criterion:

\begin{theorem}
\label{thm locally extended}The ideal $I$ is an adjoint ideal of $\;\!\Gamma$
iff $I_{P}=I_{P}\overline{\mathcal{O}_{\Gamma,P}}$ for all $P\in
\operatorname*{Sing}(\Gamma)$.
\end{theorem}

\begin{proof}
Using the notation of Remark \ref{rmk kernel rho}, we have, for $m\gg0$,%
\[
\dim_{k}\operatorname*{coker}\varrho_{m}=h^{0}\left(  \Gamma,\mathcal{F}%
\right)  =\sum_{P\in\operatorname*{Sing}(\Gamma)}\ell(I_{P}\overline
{\mathcal{O}_{\Gamma,P}}/I_{P})\text{.}%
\]
Hence, $\varrho_{m}$ is surjective iff $I_{P}\overline{\mathcal{O}_{\Gamma,P}%
}=I_{P}$ for all $P\in\operatorname*{Sing}(\Gamma)$.
\end{proof}

\begin{corollary}
\label{cor supp1}If $I$ is an adjoint ideal of $\Gamma$ and $P\in\operatorname*{Sing}%
(\Gamma)$, then $I_{P}\subsetneqq\mathcal{O}_{\Gamma,P}$.
\end{corollary}

\begin{proof}
Suppose $I_{P}=\mathcal{O}_{\Gamma,P}$. Then $I_{P}\subsetneqq I_{P}\overline
{\mathcal{O}_{\Gamma,P}}$, a contradiction to Theorem \ref{thm locally extended}.
\end{proof}

\begin{corollary}
The support of $\operatorname*{Proj}(S/I)$ contains $\operatorname*{Sing}%
(\Gamma)$.
\end{corollary}

\begin{proof}
Follows immediately from Corollary \ref{cor supp1}.
\end{proof}




\begin{theorem}
\label{thm adjoint ideal locally conductor} There is a unique largest
homogeneous ideal $\mathfrak{G}\subset S$ which satisfies%
\[
\mathfrak{G}_{P}=\mathcal{C}_{\mathcal{O}_{\Gamma,P}}\;\text{ for all }%
\;P\in\operatorname*{Sing}(\Gamma).
\]
The ideal $\mathfrak{G}$ is an adjoint ideal of \mbox{$\;\!\Gamma$} containing
all other adjoint ideals of \mbox{$\;\!\Gamma$}. In particular, $\mathfrak{G}$
is saturated and $\operatorname*{Proj}(S/\mathfrak{G})$ is supported on
$\operatorname*{Sing}(\Gamma)$.
\end{theorem}

\begin{proof}
For the conductor ideal sheaf $\mathcal{C}=\operatorname*{Ann}_{\mathcal{O}%
_{\Gamma}}(\pi_{\ast}\mathcal{O}_{\overline{\Gamma}}/\mathcal{O}_{\Gamma})$ on $\Gamma$, we
have $\mathcal{C}_{P}=\mathcal{C}_{\mathcal{O}_{\Gamma,P}}$ for all
$P\in\Gamma$. If $j:\Gamma\rightarrow\mathbb{P}_{k}^{r}$ is the inclusion,
then the graded $S$-module $\mathfrak{G}=%
{\textstyle\bigoplus\nolimits_{n\in\mathbb{Z}}}
H^{0}(\mathbb{P}_{k}^{r},j_{\ast}\mathcal{C}(n))$ associated to $j_{\ast
}\mathcal{C}$ is the unique largest homogeneous ideal with $\mathfrak{G}%
_{P}=\mathcal{C}_{\mathcal{O}_{\Gamma,P}}$ for all $P\in\operatorname*{Sing}%
(\Gamma)$. By Theorem \ref{thm locally extended} and the properties of the
conductor, $\mathfrak{G}$ is an adjoint ideal. Moreover, if $I$ is any other
adjoint ideal, then $I_{P}\subset\mathfrak{G}_{P}$ for all $P\in\Gamma$, hence
$I\subset\mathfrak{G}$.
\end{proof}

\begin{definition}
With notation as in Theorem \ref{thm adjoint ideal locally conductor}, the
ideal $\mathfrak{G}$ is called the \emph{Gorenstein adjoint ideal} of
$\;\!\Gamma$. We also write $\mathfrak{G}(\Gamma)=\mathfrak{G}$.
\end{definition}

For repeated subsequent use, we introduce the following notation:


\begin{notation}
\label{rem setup curve in PPr} Let $\Gamma\subset\mathbb{P}_{k}^{r}$ be a
curve as above with Gorenstein adjoint ideal $\;\!\mathfrak{G}$. Let $C$ be
the affine part of $\Gamma$ with respect to the chart
\[
\mathbb{A}_{k}^{r}\hookrightarrow\mathbb{P}_{k}^{r}\text{, }\left(
X_{1},...,X_{r}\right)  \mapsto\left(  1:X_{1}:\cdots:X_{r}\right)
\]
let $I(C)\subset k\left[  X_{1},...,X_{r}\right]  $ be the ideal of $C$, let
$k[C]=k[x_{1},...,x_{r}]=k\left[  X_{1},...,X_{r}\right]  /I(C)$ be its
coordinate ring, and let $\operatorname{Sing}(C)$ be its set of singular points.
\end{notation}

\begin{proposition}
\label{Cor adjoint ideal from conductor}Assume $\Gamma\subset\mathbb{P}%
_{k}^{r}$ is a curve as in Notation \ref{rem setup curve in PPr}, with affine
part $C$. Let $\;\!\overline{\mathfrak{G}}$ be the ideal of $k[C]$ obtained by
dehomogenizing $\;\!\mathfrak{G}$ with respect to $X_{0}$ and mapping the
result to $k[C]$. Then%
\[
\overline{{\mathfrak{G}}}=\mathcal{C}_{k[C]}\text{.}%
\]
If\, $\Gamma$ has no singularities at infinity\footnote{If $k$ is infinite, 
this assumption can always be achieved by a projective automorphism defined over $k$.
Otherwise,  we may have to replace $k$ by an extension field of $k$.} and $\mathcal{C}%
_{k[C]}=\left\langle g_{i}(x_{1},...,x_{r})\mid i\right\rangle _{k[C]}$ with
polynomials $g_{i}\in k[X_{1},...,X_{r}]$, then $\mathfrak{G}$ is the
homogenization of
\[
\left\langle g_{i}(X_{1},...,X_{r})\mid i\right\rangle _{k[X_{1},...,X_{r}%
]}+I(C)
\]
with respect to $X_{0}$.
\end{proposition}

\begin{proof}
The first statement is obtained by localizing at the points of $C$:
\[
\overline{{\mathfrak{G}}}_{P}=\mathcal{C}_{\mathcal{O}_{C,P}}=(\mathcal{C}%
_{k[C]})_{P}\;\text{ for each }\;P\in C\text{.}%
\]
Here, the first equality is clear from the definition of $\mathfrak{G}$ (see
Theorem \ref{thm adjoint ideal locally conductor}). The second equality holds
since forming the conductor commutes with localization since $k[C]$ is
normalization-finite (see \cite[Ch. V, \S \ 5]{ZariskiSamuel}).

The second statement of the proposition follows from the first one since there
are no singularities at infinity, $\mathfrak{G}$ is saturated, and the support
of $\mathfrak{G}$ is contained in $C$.

\end{proof}

We take a moment to specialize to plane curves.

\begin{remark}
\label{rem hilbert function adjoint ideal} Assume $\Gamma$ is a plane curve.
Then, by Max Noether's Fundamentalsatz, the maps $\varrho_{m}:\mathfrak{G}%
_{m}\rightarrow H^{0}\left(  \overline{\Gamma},\mathcal{O}_{\overline{\Gamma}%
}\left(  mH-\Delta(\mathfrak{G})\right)  \right)  $ are surjective for
\emph{all} $m$. Referring to each homogeneous polynomial in $\mathfrak{G}$ not
contained in $I(\Gamma)$ as an \emph{adjoint curve} to $\Gamma$, this means
that residual to $\Delta(\mathfrak{G})$, the adjoint curves of any degree $m$
cut out the complete linear series $\mathcal{A}_{m}=\left\vert mH-\Delta
(\mathfrak{G})\right\vert $. See \cite[\S \ 49]{vanderWaerden}.

\end{remark}


\begin{theorem}
\label{thm canonial linear series} Assume $\Gamma$ is a plane curve of degree
$n$. Then, residual to $\Delta(\mathfrak{G})$, the elements of $\mathfrak{G}%
_{n-3}$ cut out the complete canonical linear series. Equivalently,
\begin{equation}
\deg\Delta(\mathfrak{G})=2\delta(\Gamma). \label{equ:degree-delta-plane-case}%
\end{equation}

\end{theorem}

\begin{proof}
See \cite[Thm. 9]{Gorenstein}.
\end{proof}

Recall that the dimension of the canonical linear series is $\dim
\mathcal{A}_{n-3}=p(\Gamma)-1$.

\begin{remark}
Assume $\Gamma$ is a plane curve of degree $n$. If $p(\Gamma)=0$, that is,
$\Gamma$ is rational, then $\dim\mathcal{A}_{n-2}=\deg\mathcal{A}_{n-2}=n-2$.
In this case, the image of $\Gamma$ under $\mathcal{A}_{n-2}$ is a rational
normal curve $\Gamma_{n-2}\subset\mathbb{P}_{k}^{n-2}$ of degree $n-2$. Via
the birational morphism $\Gamma_{n-2}\rightarrow\Gamma$, the problem of
parametrizing $\Gamma$ is reduced to parametrizing the smooth curve
$\Gamma_{n-2}$. For the latter, we may successively decrease the degree of the
rational normal curve by $2$ via the anti-canonical linear series. This yields
an isomorphism from $\Gamma_{n-2}$ either to $\mathbb{P}^{1}$ or to a plane
conic, depending on whether $n$ is odd or even. If $\Gamma$ is defined by an
equation over a subfield $L\subset k$, then all computations considered so far
take place over the coefficient field $L$. Parametrizing the conic, however,
may require a quadratic field extension, depending on whether the conic
contains an $L$-rational point or not. See \cite{Boehm} and \cite{BDLSpar} for details.
\end{remark}


By generalizing the formula in Theorem \ref{thm canonial linear series}, we
now derive a characterization of adjoint ideals, which is also valid in the
non-planar case. We use the following notation: If $I\subset S$ is a
homogeneous ideal, write $\deg I=\deg\operatorname*{Proj}(S/I)$. That is,
$\deg I$ is $(\dim I-1)!$ times the leading coefficient of the Hilbert
polynomial of $S/I$.

\begin{lemma}
\label{lem deg double point divisor} Let
$I\subset S$ be a saturated homogeneous ideal with $I(\Gamma)\subsetneqq I$.
Then%
\[
\deg\Delta(I)\leq\deg I+\delta(\Gamma),
\]
and $I$ is an adjoint ideal of $\Gamma$ iff%
\[
\deg\Delta(I)=\deg I+\delta(\Gamma)\text{.}%
\]

\end{lemma}

\begin{proof}
Let $\operatorname*{P}_{\Gamma}(t)=(\deg\Gamma)\cdot t-p_{a}(\Gamma)+1$ be the
Hilbert polynomial of $k[\Gamma]$. Denote by $I_{\Gamma}$ the image of $I$ in
$k[\Gamma]$. Then, for $m\gg0$,
\[
\deg I=\dim_{k}(S_{m}/I_{m})=\dim_{k}(k[\Gamma]_{m}/(I_{\Gamma})_{m}%
)=P_{\Gamma}(m)-\dim_{k}(I_{\Gamma})_{m}\text{.}%
\]
Moreover, by Remark \ref{rmk kernel rho}, $h^{0}\left(  \overline{\Gamma
},\mathcal{O}_{\overline{\Gamma}}\left(  mH-\Delta(I)\right)  \right)
=\dim_{k}(I_{\Gamma})_{m}+h^{0}\left(  \Gamma,\mathcal{F}\right)  \geq\dim
_{k}(I_{\Gamma})_{m}$ for $m\gg0$. Hence, by Riemann-Roch, for $m\gg0$, we
have
\begin{align*}
(\deg\Gamma)\cdot m-\deg\Delta(I)  &  =\deg\left\vert mH-\Delta(I)\right\vert
=\dim\left\vert mH-\Delta(I)\right\vert +p(\Gamma)\\
&  \geq\dim_{k}(I_{\Gamma})_{m}-1+p(\Gamma)\\
&  =\operatorname*{P}\nolimits_{\Gamma}(m)-\deg I-1+p(\Gamma)\\
&  =(\deg\Gamma)\cdot m-\delta(\Gamma)-\deg I\text{.}%
\end{align*}
Here, we use that $\left\vert mH-\Delta(I)\right\vert $ is nonspecial for
large $m$ by reason of its degree.
Equality holds iff $\varrho_{m}$ is surjective.
\end{proof}

\begin{remark}
In the case where $\Gamma$ is a plane curve and $I=\mathfrak{G}$ is its
Gorenstein adjoint ideal, Lemma \ref{lem deg double point divisor} shows that
Equation \eqref{equ:degree-delta-plane-case} may be rewritten as
\begin{equation}
\deg\mathfrak{G}=\delta(\Gamma)\text{.} \label{equ:degree-delta-plane-case-ll}%
\end{equation}

\end{remark}

Note that \eqref{equ:degree-delta-plane-case} and
\eqref{equ:degree-delta-plane-case-ll} may not hold in the non-planar case:

\begin{example}
[{\negthinspace\negthinspace{\cite[Example~5.2.5]{DeJong}}}]Let $\Gamma
\subset\mathbb{P}_{\mathbb{C}}^{3}$ be the image of the parametrization
\[
\mathbb{P}_{\mathbb{C}}^{1}\longrightarrow\mathbb{P}_{\mathbb{C}}^{3}\text{,
}(s:t)\mapsto(s^{5}:t^{3}s^{2}:t^{4}s:t^{5}).
\]
Then $\Gamma$ has exactly one singularity at \mbox{$(1:0:0:0)$}. Furthermore,
$p(\Gamma)=0$ and $p_{a}(\Gamma)=2$, hence $\delta(\Gamma)=2$. However,
$\mathfrak{G}=\left\langle X_{1},X_{2},X_{3}\right\rangle \subset
\mathbb{C}[X_{0},\ldots,X_{3}]$, hence $\deg\mathfrak{G}=1$.
\end{example}

\begin{remark}
If $\Gamma\subset\mathbb{P}_{k}^{r}$ is any curve as in Notation
\ref{rem setup curve in PPr}, with affine part $C$ and no singularities at
infinity, then it follows from Proposition
\ref{Cor adjoint ideal from conductor} that
\[
\deg\mathfrak{G}=\dim_{k}\left(  k[C]/\mathcal{C}_{k[C]}\right)
={\textstyle\sum\nolimits_{P\in\operatorname*{Sing}(C)}}\dim_{k}%
(\mathcal{O}_{C,P}/\mathcal{C}_{\mathcal{O}_{C,P}}).
\]
\end{remark}


\begin{lemma}
If $\operatorname*{char}k=0$, then $\dim_{k}(\mathcal{O}_{\Gamma
,P}/\mathcal{C}_{\mathcal{O}_{\Gamma,P}})\leq\delta_{P}(\Gamma)$ for any point
$P\in\Gamma$.
\end{lemma}

\begin{proof}
Since normalization commutes with base change, this follows from the case $k=\mathbb C$
proved in \cite[2.4]{Greuel1982}.
\end{proof}
Now recall that a point $P\in\operatorname{Sing}(\Gamma)$ is called a
\emph{Gorenstein singularity} if
\[
\dim_{k}(\mathcal{O}_{\Gamma,P}/\mathcal{C}_{\mathcal{O}_{\Gamma,P}}%
)=\delta_{P}(\Gamma)\text{.}%
\]

\begin{example}
Plane curve singularities are Gorenstein (see, for example,
\cite[Corollary~5.2.9]{DeJong}).
\end{example}

\begin{corollary}
\label{cor delta gorenstein adjoint ideal}
We have:

\begin{enumerate}
\item If $\operatorname*{char}k=0$, then $\deg\mathfrak{G}\leq\delta(\Gamma)$.

\item If $\Gamma$ has only Gorenstein singularities, then
\[
\deg\mathfrak{G}=\delta(\Gamma)\;\text{ and }\;\deg\Delta(\mathfrak{G}%
)=2\delta(\Gamma)\text{.}%
\]

\end{enumerate}
\end{corollary}

\begin{proof}
This is clear from the discussion above.

\end{proof}

In the case of arbitrary singularities, we will make use of the equality
\[
\deg\mathfrak{G}=\deg\Delta(\mathfrak{G})-\delta(\Gamma)
\]
to compute $\deg\mathfrak{G}$ without actually knowing $\mathfrak{G}$, and
apply this in the final verification step of our modularized adjoint ideal
algorithm. To this end, if $k=\mathbb{C}$ and $\Gamma$ is defined over the
rationals, we will present a modular approach to computing $\deg
\Delta(\mathfrak{G})$, and we will use standard techniques to compute
$\delta(\Gamma)$. In fact, for the latter, first note that the delta invariant
of $\Gamma$ differs from that of a plane model of $\Gamma$ by the quantity
$p_{a}(\Gamma)-\binom{\deg\Gamma-1}{2}$. The delta invariant of a plane curve,
in turn, can be computed locally at the singular points, either from the
semigroups of values of the analytic branches of the singularity (see
\cite{DeJong}, \cite{GLSsing}), or from a formula relating the local delta
invariant to the Milnor number (see Remark \ref{rem:comp-inv-plane-curves} in
Section \ref{sec summary} below).

\begin{remark}
\label{lem deg delta} Note that computing $\deg\Delta(\mathfrak{G})$ also
means to compute the dimension $\dim_{k}\big(  \overline{k[C]}/\mathcal{C}%
_{k[C]}\big)  $: Given $\Gamma\subset\mathbb{P}_{k}^{r}$ as in Notation
\ref{rem setup curve in PPr}, with affine part $C$ and no singularities at
infinity, we have
\begin{align*}
\deg\Delta(\mathfrak{G})  &  =\delta(\Gamma)+\deg\mathfrak{G}\\
&  =\dim_{k}\big(  \overline{k[C]}/k[C]\big)  +\dim_{k}\big(
k[C]/\mathcal{C}_{k[C]}\big) \\
&  =\dim_{k}\big(  \overline{k[C]}/\mathcal{C}_{k[C]}\big)  .
\end{align*}

\end{remark}


We are now ready to address the computation of the Gorenstein adjoint ideal.
Using Proposition \ref{Cor adjoint ideal from conductor}, one way of finding
$\mathfrak{G}$ is to apply the global algorithm presented in Section
\ref{sec ideal quotient} below, starting from the normalization $\overline
{k[C]}$. The normalization, in turn, can be found by combining the minimal
local contributions to $\overline{k[C]}$ at the singular points via
Proposition \ref{prop:local-to-global-I}. As it turns out, however, it is more
efficient to directly compute local Gorenstein adjoint ideals at the singular
points, and get $\mathfrak{G}$ as their intersection. This will be discussed
in Sections \ref{Sec Global from local} and \ref{Sec General local approch}.

\begin{remark} In applications, $\Gamma\subset\mathbb{P}_{k}^{r}$ is often 
defined over a perfect subfield $k'\subset k$ (for example, $k'=\QQ$ and 
$k=\mathbb C$). In such a situation, by base change, 
$\delta(\Gamma)=\delta(\Gamma(k'))$. Moreover, since the 
algorithms in Sections \ref{Sec Global from local} and \ref{Sec General local approch}
rely on Gr\"obner bases, and Buchberger's algorithm for computing Gr\"obner 
bases does not leave the ground field, $\mathfrak{G}(\Gamma)=\mathfrak{G}
(\Gamma(k'))K[X_0,\ldots,X_n]$, and generators can be found by computations over $k'$. 
\end{remark}

\section{Global approaches\label{Sec Global approach}}

\subsection{Computing the conductor via the trace
matrix\label{sec linear algebra}}

We will require some facts from classical ideal theory (see \cite[Ch.
V]{ZariskiSamuel} for details and proofs): Let $R$ be an integral domain, and
let $K=\operatorname*{Q}\left(  R\right)  $ be its quotient field. A
\emph{fractionary ideal} of $R$ is an $R$-submodule $\mathfrak{b}$ of $K$
admitting a common denominator: there is an element $0\neq d\in R$ such that
$d\ \! \mathfrak{b}\subset R$.

\begin{example}
The extensions $A_{i}$ computed by the normalization algorithms from Section
\ref{sec normalization} are fractionary ideals of the given affine domain $A$.
\end{example}

If $\mathfrak{b}, \mathfrak{b^{\prime}}$ are two fractionary ideals of $R$,
with $\mathfrak{b}^{\prime}$ nonzero, then $\mathfrak{b} :\mathfrak{b^{\prime
}}=\{ z\in K\mid z\;\!\mathfrak{b^{\prime}}\subset\mathfrak{b}\}$ is a
fractionary ideal of $R$ as well. A fractionary ideal $\mathfrak{b}$ of $R$ is
\emph{invertible} if there is a fractionary ideal $\mathfrak{b^{\prime}}$ of
$R$ such that $\mathfrak{b}\cdot\mathfrak{b^{\prime}}=R$. In this case,
$\mathfrak{b^{\prime}}$ is uniquely determined and equal to $R:\mathfrak{b}$.

Suppose in addition that $R$ is normal. Let $K^{\prime}$ be a finite separable
extension of $K$, and let $R^{\prime}$ be an integral extension of $R$ such
that $K^{\prime}=\operatorname*{Q}\left(  R^{\prime}\right)  $. Moreover, let
\[
\operatorname{Tr}_{K^{\prime}/K}:K^{\prime}\rightarrow K\text{, }%
z\mapsto{\displaystyle\sum\limits_{g\in\operatorname*{Gal}\left(  K^{\prime
}/K\right)  }}g(z)\text{,}%
\]
be the \emph{trace map}. Then the \emph{complementary module}
\[
\mathfrak{C}_{R^{\prime}/R}:=\left\{  z\in K^{\prime}\mid\operatorname{Tr}%
_{K^{\prime}/K}\left(  zR^{\prime}\right)  \subset R\right\}
\]
of $R^{\prime}$ with respect to $R$ is a \emph{fractionary ideal} of
$R^{\prime}$ containing $R^{\prime}$. Hence, the \emph{different}\textbf{ }%
\begin{align*}
\mathfrak{D}_{R^{\prime}/R}  &  =R^{\prime}:\mathfrak{C}_{R^{\prime}%
/R}=\left\{  z\in K^{\prime}\mid z\;\!\mathfrak{C}_{R^{\prime}/R}\subset
R^{\prime}\right\} \\
&  =\left\{  z\in K^{\prime}\mid zx\in R^{\prime}\text{ for all }x\in
K^{\prime}\text{ with }\operatorname{Tr}_{K^{\prime}/K}\left(  xR^{\prime
}\right)  \subset R\right\}
\end{align*}
of $R^{\prime}$ over $R$ is a nonzero ideal of $R^{\prime}$.

Now, keeping our assumptions, we focus on the case where $R$ is a Dedekind
domain, and where $R^{\prime}$ is the integral closure of $R$ in $K^{\prime}$.
Then $R^{\prime}$ is a Dedekind domain as well, which implies that every
nonzero fractionary ideal of $R^{\prime}$ is invertible. On the other hand, by
the primitive element theorem, there is an element $y\in R^{\prime}$ with
$K^{\prime}=K(y)$. Denote by $f(Y)\in K[Y]$ the minimal polynomial of $y$ over
$K$. Then, as shown in \cite[Ch. V]{ZariskiSamuel},
\[
f^{\prime}(y)R^{\prime}=\mathcal{C}_{R^{\prime}/R[y]}\mathfrak{D}_{R^{\prime
}/R}\text{,}%
\]
hence%
\begin{equation}
\mathcal{C}_{R^{\prime}/R[y]}=f^{\prime}(y)\mathfrak{C}_{R^{\prime}/R}%
\text{{}}. \label{equ conductor}%
\end{equation}

We now fix the following setup:

\begin{notation}
\label{rmk setup}
Let $k$ be a perfect field. Let $\Gamma\subset\mathbb{P}_{k}^{2}$ 
be a plane curve of degree $n$ defined
by an irreducible polynomial $F\in k[X,Y,Z]$. Suppose that $\Gamma$ has no
singularities at infinity with respect to the affine chart
\[
\mathbb{A}_{k}^{2}\hookrightarrow\mathbb{P}_{k}^{2}\text{, }\left(
X,Y\right)  \mapsto\left(  1:X:Y\right)  ,
\]
and that the equation $f\in$ $k[X,Y]$ of the affine part $C$ of $\Gamma$ is
monic in $Y$.
\end{notation}

Write $k[C]=k[x,y]=k[X,Y]/\langle f(X,Y)\rangle$ for the affine coordinate
ring of $C$ and
\[
k(C)=k(x,y)=k(X)[Y]/\langle f(X,Y)\rangle
\]
for its function field. Then $x$ is a separating transcendence basis of $k(C)$
over $k$, and $y$ is integral over $k[x]$, with integral equation $f(x,y)=0$.
In particular, $k[C]$ is integral over $k[x]$, which implies that
$\overline{k[C]}$ coincides with the integral closure $\overline{k[x]}$ of
$k[x]$ in $k(C)$. Furthermore, $\overline{k[C]}$ is a free $k[x]$-module of
rank
\[
n:=\deg_{y}(f)=[k(C):k(x)]\text{.}%
\]

\begin{definition}
An {\emph{integral basis}} for $\overline{k[C]}$ is a set $b_{0},\dots
,b_{n-1}$ of free generators for $\overline{k[C]}$ over $k[x]$:
\[
\overline{k[C]}=k[x]b_{0}\oplus\cdots\oplus k[x]b_{n-1}\text{.}%
\]

\end{definition}

\begin{remark}
\label{rem:spec-int-basis}Since $k(C)=k(x,y)=k(X)[Y]/\!\left\langle
f\right\rangle $, any element $\alpha\in k(C)$ can be represented as a
polynomial in $k(X)[Y]$ of degree less than $n=\deg f$. Hence, one can
associate to $\alpha$ a well-defined degree $\deg_{y}\left(  \alpha\right)  $
in $y$ and a smallest common denominator in $k[x]$ of the coefficients of
$\alpha$. In particular, $\overline{k[C]}$ has an integral basis $(b_{i})$ in
triangular form, that is, with $\deg_{y}(b_{i})=i$, for $i=0,...,n-1$. If not
stated otherwise, all integral bases will be of this form. In principle, such
a basis can be found by applying one of the normalization algorithms discussed
earlier. However, in the characteristic zero case, methods relying on Puiseux
series techniques are much more efficient (see \cite{BDLS} and
\cite{vanHoeijIntegralBasis}).
\end{remark}

\begin{example}
\label{ex a4 E8-ib} An integral basis for the curve from Example
\ref{ex a4 E8} is given below:%
\[
1,y,\frac{y(y-1)}{x},\frac{y(y-1)^{2}}{x^{2}},\frac{y^{2}(y-1)^{2}}{x^{3}%
}\text{.}%
\]

\end{example}

Using Proposition \ref{Cor adjoint ideal from conductor} and Equation
\eqref{equ conductor}, with $R=k[x]$, $R^{\prime}=\overline{k[C]}$, $K=k(x)$,
and $K^{\prime}=k(C)$, we get Algorithm \ref{alg conductor linear algebra}.

\begin{algorithm}[h]
\caption{Gorenstein adjoint ideal via linear algebra (see \cite{Mnuk})}
\label{alg conductor linear algebra}
\begin{algorithmic}[1]
\REQUIRE A plane curve $\Gamma$ with affine part $C$ as in Notation \ref{rmk setup}.
\ENSURE The Gorenstein adjoint ideal $\mathfrak{G}$ of $\Gamma$.
\STATE Compute an integral basis $\left(  b_{i}\right)  _{i=0,...,n-1}$ for
$\overline{k[  C]  }$.
\STATE Compute the (symmetric and invertible) trace matrix%
\[
T=\left(  \operatorname{Tr}_{k\left(  C\right)  /k\left(  x\right)  }\left(
b_{i}b_{j}\right)  \right)  _{i,j=0,...,n-1}\in k(  x)  ^{n\times
n}.
\]
\STATE Compute a decomposition $L\cdot R=P\cdot T$, where $L$ is left
triangular matrix with diagonal entries equal to one, $R$ is a right
triangular matrix, and $P$ is a permutation matrix.
\STATE For $j=0,...,n-1$, use forward and backward substitution to compute
\[
\eta_{j}=\sum\limits_{i=0}^{n-1}s_{ij}b_{i},
\]
where $\left(  s_{ij}\right)  =T^{-1}$. The $\eta_{j}$ are $k[x]$-module
generators for $\mathfrak{C}_{\overline{k[  C]}/k[  x]  }$.
By \eqref{equ conductor}, $\mathcal{C}_{k\left[  C\right] }= \langle\frac{\partial
f}{\partial Y}\left(  x,y\right)  \eta_{j} \mid j=0,...,n-1\rangle$.
\STATE Let $\mathcal{C}$ be the ideal of $k[X,Y]$ generated by representatives
of minimal $y$-degree of the $\frac{\partial f}{\partial Y}\left(  x,y\right)
\eta_{j}$, $j=0,...,n-1$.
\RETURN the homogenization of $\mathcal{C}$ with respect to $X_{0}$.
\end{algorithmic}
\end{algorithm}

\begin{remark}
\label{rmk subfield}
To compute an integral basis via Puiseux series in the characteristic zero
case, we temporarily may have to pass to an algebraic extension field of $k$.
\end{remark}

\begin{example}
\label{ex A4} The curve $\Gamma\subset\mathbb{P}_{\mathbb{C}}^{2}$ from
Example \ref{ex a4 E8} with affine equation
\[
X^{5}-Y^{2}\left(  1-Y\right)  ^{3}=0
\]
has a singularity of type $A_{4}$ at $\left(  0,0\right)  $ and a $3$-fold
point of type $E_{8}$ at $\left(  0,1\right)  $. From the integral basis
\[%
\begin{tabular}
[c]{lllll}%
$b_{0}=1$, & $b_{1}=y$, & $b_{2}=\frac{y\left(  y-1\right)  }{x}$, &
$b_{3}=\frac{y\left(  y-1\right)  ^{2}}{x^{2}}$, \ and & \negthinspace
\negthinspace\negthinspace\negthinspace\ $b_{4}=\frac{y^{2}\left(  y-1\right)
^{2}}{x^{3}}$%
\end{tabular}
\]
given in Example \ref{ex a4 E8-ib}, we compute the trace matrix%
\[
T=\left(
\begin{array}
[c]{ccccc}%
5 & 3 & 0 & 0 & 0\\
3 & 3 & 0 & 0 & -5x^{2}\\
0 & 0 & 0 & -5x^{2} & -3x\\
0 & 0 & -5x^{2} & -3x & 0\\
0 & -5x^{2} & -3x & 0 & 0
\end{array}
\right)  ,
\]
which yields by forward and backward substitution%
\[
\mathcal{C}_{\mathbb{C}[C]}=\left\langle x^{3},\;x^{2}\left(  y-1\right)
,\;xy\left(  x-1\right)  ,\;y\left(  y-1\right)  ^{2}\right\rangle
_{\mathbb{C}[C]}\text{.}%
\]
Homogenization (and primary decomposition) gives%
\[
\mathfrak{G}=\left\langle Y,X^{2}\right\rangle \cap\left\langle X^{3},X\left(
Y-Z\right)  ,\left(  Y-Z\right)  ^{2}\right\rangle \text{.}%
\]
Writing $\mathfrak{G}$ as the intersection of simpler ideals corresponding to
the singularities motivates the local to global approach discussed in Sections
\ref{Sec Global from local} and \ref{Sec General local approch} below, where
$\mathfrak{G}$ will be found as the intersection of local Gorenstein ideals.
\end{example}

\pagebreak[4]

\subsection{Computing the adjoint ideal via ideal
quotients\label{sec ideal quotient}}

The algorithm presented in what follows relies on normalization and ideal
quotients. It is not limited to plane curves.

\begin{proposition}
\label{prop adjoint ideal and conductor} Let $\Gamma\subset\mathbb{P}_{k}^{r}$
be a curve with affine part $C$ as in Notation \ref{rem setup curve in PPr}.
Write $\overline{k[C]}=\frac{1}{d}U$, where $U\subset k[C]$ is an ideal and
$d\in U$ is nonzero. Then the conductor is%
\[
\mathcal{C}_{k[C]}=\left\langle d\right\rangle _{k[C]} :U\text{.}%
\]

\end{proposition}

\begin{proof}
By definition,%
\begin{align*}
\mathcal{C}_{k[C]}  &  =\left\{  s\in k[C]\mid s\cdot\overline{k[C]}\subset
k[C]\right\} \\
&  =\left\{  s\in k[C]\mid s\cdot g\in\left\langle d\right\rangle \text{ for
all }g\in U\right\} \\
&  =\left\langle d\right\rangle _{k[C]} :U.
\end{align*}

\vspace{-0.5cm}
\end{proof}

Using once more Proposition \ref{Cor adjoint ideal from conductor}, we get
Algorithm \ref{Alg ideal quotient gorenstein adjoint ideal}. 

\begin{algorithm}[H]
\caption{Gorenstein adjoint ideal via ideal quotients}
\label{Alg ideal quotient gorenstein adjoint ideal}
\begin{algorithmic}[1]
\REQUIRE A curve $\Gamma\subset\mathbb{P}_{k}^{r}$ with affine part $C$ as in Notation
\ref{rem setup curve in PPr} and no singularities at infinity.
\ENSURE The Gorenstein adjoint ideal $\mathfrak{G}$ of $\Gamma$.
\STATE Normalization: Compute polynomials $d, a_{0},\dots, a_{s}\in k[
X_{1},...,X_{r}]  $ such that the fractions $\frac{a_{i}(
x_{1},...,x_{r})  }{d(  x_{1},...,x_{r})  }$ generate
$\overline{k  [C]  }$ as a $k[  C]  $-module.
\STATE Compute the ideal quotient%
\[
\mathcal{C}=(\left\langle d\right\rangle +I(  C)  ):(\left\langle
a_{0},...,a_{s}\right\rangle +I(  C)  )\subset k[
X_{1},...,X_{r}]  .
\]
\vspace{-0.4cm}
\RETURN the homogenization of $\mathcal{C}$ with respect to $X_{0}$.
\end{algorithmic}
\end{algorithm}


\begin{example}
In Example \ref{ex a4 E8},
\[%
\begin{tabular}
[c]{llll}%
$a_{0}=X^{3}$, & $a_{1}=X^{2}Y(Y-1)$, & $a_{2}=XY(Y-1)^{2}$, & $a_{3}%
=Y^{2}(Y-1)^{2}$,
\end{tabular}
\
\]
and $d=X^{3}$. Hence,
\[
\left\langle d,f\right\rangle :\left\langle a_{0},...,a_{3},f\right\rangle
=\left\langle X^{3},\;X^{2}\left(  Y-1\right)  ,\;XY\left(  Y-1\right)
,\;Y\left(  Y-1\right)  ^{2}\right\rangle .
\]

\end{example}

\section{A Local to global Approach\label{Sec Global from local}}

In this section, motivated by the local to global approach for normalization,
we introduce local Gorenstein adjoint ideals of a given curve and show how to
find the Gorenstein adjoint ideal $\mathfrak{G}$ as their intersection.
Together with the algorithm presented in the next section, where we will show
how to compute the local ideals, this yields a local to global approach for
finding $\mathfrak{G}$. As we will see in Section
\ref{Sec examples and comparisons}, this approach is per se faster than the
algorithms discussed so far. In addition, it is well-suited for parallel
computations. 

We fix the following setup:

\begin{notation}
\label{rem setup curve in PPr 2} Let $\Gamma\subset\mathbb{P}_{k}^{r}$ be an
integral non-degenerate projective curve, and let $S$ be the homogeneous coordinate
ring of $\mathbb{P}_{k}^{r}$.
\end{notation}

\begin{definition}
Let $W\subset\operatorname*{Sing}(\Gamma)$ be a set of singular points of
$\Gamma$. The \emph{local Gorenstein adjoint ideal} of $\Gamma$ at $W$ is
defined to be the largest homogeneous ideal $\mathfrak{G}(W) \subset S$ which
satisfies%
\begin{equation}
\mathfrak{G}(W)_{P}=\mathcal{C}_{\mathcal{O}_{\Gamma,P}}\;\text{ for all
}\;P\in W. \label{eq:loc-gor-ideal}%
\end{equation}
For a single point $P\in\operatorname*{Sing}(\Gamma)$, we write $\mathfrak{G}%
(P):=\mathfrak{G}(\{P\})$.
\end{definition}

\begin{remark}
\label{rem:local-adjoint-ideal}Since $\mathfrak{G}(W)$ is the largest
homogeneous ideal satisfying \eqref{eq:loc-gor-ideal}, it is saturated and
$\operatorname*{Proj}(S/\mathfrak{G}(W))$ is supported on $W$.

\end{remark}


\begin{proposition}
\label{prop global from local} Let $W\subset\operatorname*{Sing}(\Gamma)$.
Then
\[
\mathfrak{G}(W)={\textstyle\bigcap\nolimits_{P\in W}}\mathfrak{G}(P)\text{.}%
\]

\end{proposition}

\begin{proof}
This is immediate from the definition: If $\mathfrak{G}^{\prime}%
:={\textstyle\bigcap\nolimits_{P\in W}}\mathfrak{G}(P)$, then
$\operatorname*{Proj}(S/\mathfrak{G}^{\prime})$ and $\operatorname*{Proj}%
(S/\mathfrak{G}(W))$ have the same support $W$, and
\[
\mathfrak{G}_{Q}^{\prime}=\mathfrak{G}(Q)_{Q}=\mathcal{C}_{\mathcal{O}%
_{\Gamma,Q}}=\mathfrak{G}(W)_{Q}%
\]
for all $Q\in W$, hence $\mathfrak{G}(W)=\mathfrak{G}^{\prime}$.

\end{proof}

Proposition \ref{prop global from local} yields Algorithm
\ref{alg global from local}.

\begin{algorithm}[H]
\caption{Gorenstein adjoint ideal, local to global}
\label{alg global from local}
\begin{algorithmic}[1]
\REQUIRE A curve $\Gamma\subset\mathbb{P}_{k}^{r}$  as in Notation
\ref{rem setup curve in PPr 2}.
\ENSURE The Gorenstein adjoint ideal $\mathfrak{G}$ of $\Gamma$.
\STATE Compute $\operatorname*{Sing}(\Gamma)=\{ P_{1},...,P_{s}\}$.
\STATE Apply Algorithm \ref{alg local integral quotient} in Section
\ref{Sec General local approch} below to compute
$\mathfrak{G}\left(  P_{i}\right)  $ for all $i$.
\RETURN ${\bigcap\nolimits_{i=1}^{s}}\mathfrak{G}\left(  P_{i}\right)  $.
\end{algorithmic}
\end{algorithm}

\begin{remark}
It is clear from Proposition \ref{prop global from local} that we may choose
any partition $\operatorname*{Sing}(\Gamma)={\textstyle\bigcup\nolimits_{i=1}%
^{s}}W_{i}$ of $\operatorname*{Sing}(\Gamma)$ and have
\[
\mathfrak{G}={\textstyle\bigcap\nolimits_{i=1}^{s}}\mathfrak{G}(W_{i})\text{.}%
\]
This is useful in that for some subsets $W_{i}$, specialized approaches or a
priori knowledge may ease the computation of $\mathfrak{G}(W_{i})$. In Section
\ref{sec summary}, we will present some ideas in this direction for plane curves.

\end{remark}



\section{Computing local adjoint ideals\label{Sec General local approch}}

In this section, we modify Algorithm
\ref{Alg ideal quotient gorenstein adjoint ideal} so that it computes the
local Gorenstein adjoint ideal at a point $P$ from a minimal local
contribution at $P$ via ideal quotients.

We consider a curve  
$\Gamma\subset\mathbb{P}_{k}^{r}$ as in Notation \ref{rem setup curve in PPr}
with affine part\footnote{To cover all singular points of $\Gamma$, we may have 
to choose affine charts other than that considered in Notation \ref{rem setup curve in PPr}.} 
$C$ and a point $P\in\operatorname*{Sing}(C)$. Let $\frac{1}{d}U$ 
be the minimal local contribution to $\overline{k[C]}$ at $P$; so $U\subset k[${$C]$} 
is an ideal and $d\in U$ is nonzero.

\begin{proposition}
\label{prop from local contribution}With notation as above, and given $Q\in
C$, we have
\[
(\left\langle d\right\rangle _{k[C]}:U)_{Q}=\left\{
\begin{tabular}
[c]{ll}%
$\mathcal{C}_{\mathcal{O}_{C,Q}}$ & if $Q=P$,\\
$\mathcal{O}_{C,Q}$ & if $Q\neq P$.
\end{tabular}
\right.
\]

\end{proposition}

\begin{proof}
By the minimality assumption, we have%
\[
\left(  \frac{1}{d}U\right)  _{Q}=\left\{
\begin{tabular}
[c]{ll}%
$\overline{\mathcal{O}_{C,Q}}$ & if $Q=P$,\\
$\mathcal{O}_{C,Q}$ & if $Q\neq P$.
\end{tabular}
\ \right.
\]
The claim follows since localization commutes with forming the conductor:%
\[
\left(  \left\langle d\right\rangle _{k[C]}:U\right)  _{Q}=\left(
\mathcal{C}_{\left(  \frac{1}{d}U\right)  \;/\;k[C]}\right)  _{Q}%
=\mathcal{C}_{\left(  \frac{1}{d}U\right)  _{Q}\;/\;k[C]_{Q}}\text{.}%
\]

\vspace{-0.5cm}
\end{proof}

Now, we argue as in the proof of Proposition
\ref{Cor adjoint ideal from conductor}: From Proposition
\ref{prop from local contribution} and Remark \ref{rem:local-adjoint-ideal},
it follows that $\left\langle d\right\rangle _{k[C]} :U$ coincides with the
ideal obtained by dehomogenizing $\mathfrak{G}(P)$ with respect to $X_{0}$ and
mapping the result to $k[C]$. Hence, since $\mathfrak{G}(P)$ is saturated,
Algorithm \ref{alg local integral quotient} below indeed computes
$\mathfrak{G}(P)$.

\begin{algorithm}[H]
\caption{Local Gorenstein adjoint ideal from local contribution}
\label{alg local integral quotient}
\begin{algorithmic}[1]
\REQUIRE A curve $\Gamma\subset\mathbb{P}_{k}^{r}$ with affine part $C$
as in Notation \ref{rem setup curve in PPr}
and a point $P\in\operatorname{Sing}(  C )  $.
\ENSURE The local Gorenstein adjoint ideal $\mathfrak{G}(P)  $ of
$C$.
\STATE Compute polynomials $d, a_{0},\dots, a_{s}\in k\left[  X_{1}%
,...,X_{r}\right]  $ such that the fractions $\frac{a_{i}\left(
x_{1},...,x_{r}\right)  }{d\left(  x_{1},...,x_{r}\right)  }$ generate the
minimal local contribution to $\overline{k\left[  C\right]  }$ at $P$ as a
$k\left[  C\right]  $-module.
\STATE Compute the ideal quotient%
\[
\mathcal{C}=(\left\langle d\right\rangle +I\left(  C\right)  ):(\left\langle
a_{0},...,a_{s}\right\rangle +I\left(  C\right)  )\subset k\left[
X_{1},...,X_{r}\right]  .
\]
\vspace{-0.4cm}
\RETURN the homogenization of $\mathcal{C}$ with respect to $X_{0}$.
\end{algorithmic}
\end{algorithm}



\begin{example}
\label{exa:local-ideal-quotient} We compute the local Gorenstein adjoint
ideals for the curve given in Example \ref{ex a4 E8} with affine equation
\[
X^{5}-Y^{2}\left(  1-Y\right)  ^{3}=0\text{.}%
\]
For the $A_{4}$-singularity $P_{1}$, we found
\[
d_{1}=x^{2}\;{\text{ and }}\;U_{1}=\left\langle x^{2},\text{ }y(y-1)^{3}%
\right\rangle _{\mathbb{C}[C]}\text{,}%
\]
so that
\[
\mathfrak{G}(P_{1})=\left\langle X^{2},\text{ }Y\right\rangle \text{.}%
\]
For the $E_{8}$ singularity $P_{2}$, we observed that
\[
d_{2}=x^{3}\;{\text{ and }}\;U_{2}=\left\langle x^{3},\text{ }x^{2}%
y^{2}\left(  y-1\right)  ,\text{ }y^{2}\left(  y-1\right)  ^{2}\right\rangle
_{{\mathbb{C}[C]}}\text{,}%
\]
leading to
\[
\mathfrak{G}(P_{2})=\left\langle X^{3},\text{ }X(Y-Z),(Y-Z)^{2}\right\rangle
\text{.}%
\]
Note that $\mathfrak{G}(P_{1})$ and $\mathfrak{G}(P_{2})$ are the ideals
already obtained in Example \ref{ex A4}.
\end{example}

\section{Improvements to the local strategy for plane
curves\label{sec summary}}

In this section, we focus on the case of a plane curve $\Gamma$ with affine
part $C=V(f)$ and $\operatorname{Sing}(\Gamma)=\operatorname{Sing}(C)$ as
in Notation \ref{rmk setup}. \emph{For simplicity of the presentation, we
suppose throughout the section that our ground field $k=\mathbb{C}$}.

As explained in Section \ref{Sec Global from local}, the Gorenstein adjoint
ideal $\mathfrak{G}$ can be computed as the intersection of local Gorenstein
ideals via a partition of $\operatorname*{Sing}(C)$. To begin with, consider
the following partition:
\begin{equation}
\operatorname{Sing}(C)=W_{2}\cup W_{3}\cup\cdots\cup W_{r}\cup W^{\prime},
\label{part-sing}%
\end{equation}
where, for all $i$, $W_{i}$ denotes the locus of ordinary $i$-fold points
(ordinary multiple points of multiplicity $i$), and where $W^{\prime}$
collects the remaining singularities of $C$. In particular, $W_{2}$ is the set
of nodes of $C$.

\begin{lemma}
Let \mbox{$P\in\operatorname*{Sing}(C)$}, and let
\mbox{$\mathfrak{m}_P\subset k[X,Y]$} be the corresponding maximal ideal. If
$P$ is an ordinary $i$-fold point of $C$, then
\[
\mathfrak{G}(P)=\mathfrak{m}_{P}^{i-1}\text{.}%
\]

\end{lemma}

\begin{proof}
Since $C$ is a plane curve and $P$ is an ordinary $i$-fold point of $C$, 
the conductor $\mathcal{C}_{\mathcal{O}_{C,P}}
=\mathfrak{m}_{C,P}^{i-1}$, where $\mathfrak{m}_{C,P} $ is the maximal ideal
of $\mathcal{O}_{C,P}$ (see \cite{Matlis}, \cite{GrecoValabrega}). The result
follows from the very definition of $\mathfrak{G}(P)$.

\end{proof}

Applying the lemma to the partition \eqref{part-sing}, we get the intersection
of ideals
\begin{equation}
\mathfrak{G}=I\left(  W_{2}\right)  \cap I\left(  W_{3}\right)  ^{2}\cap
\cdots\cap I\left(  W_{r}\right)  ^{r}\cap\mathfrak{G}(W^{\prime})\text{.}
\label{part-sing-II}%
\end{equation}
Hence, in the case where $\Gamma$ is known to have ordinary multiple points as
singularities only (that is, $W^{\prime}=\emptyset$), we can compute
$\mathfrak{G}$ in a very efficient way by using Algorithm
\ref{alg ordinary multiple points} below (see \cite{Boehm}).

\begin{algorithm}[H]
\caption{Gorenstein adjoint ideal, ordinary multiple points only}
\label{alg ordinary multiple points}
\begin{algorithmic}[1]
\REQUIRE A plane curve $\Gamma$ of degree $n$ with defining polynomial $F$
as in Notation \ref{rmk setup} with only ordinary multiple points.
\ENSURE The Gorenstein adjoint ideal $\mathfrak{G}$ of $\Gamma$.
\STATE $J_{1}:=\left\langle \frac{\partial F}{\partial X},\frac{\partial
F}{\partial Y},\frac{\partial F}{\partial Z}\right\rangle$ \quad\quad\quad\quad(the ideal
defining $\operatorname{Sing}(\Gamma)$)
\STATE  $i:=1$
\WHILE {$({J_{i}}:\left\langle X,Y,Z\right\rangle ^{\infty})\neq
\left\langle 1\right\rangle $}
\STATE $i:=i+1$
\STATE $J_{i}:=\left\langle \frac{\partial^{j+l+m}F}{\partial X^{j}\partial
Y^{l}\partial Z^{m}}\mid j+l+m=i\text{, }j,l,m\in\mathbb{N}_{0}\right\rangle
$
\ENDWHILE
\STATE $B:=\left\langle X,Y,Z\right\rangle ^{n-i}$
\WHILE {$i>0$}
\STATE $I_{i}:=\left(  J_{i-1}:B^{\infty}\right) $ \quad\quad\quad(the ideal of the
$i$-fold points of $\Gamma$)
\STATE $B:=((B\cap I_{i}^{i-1}):\left\langle X,Y,Z\right\rangle ^{\infty})$
\STATE $i:=i-1$
\ENDWHILE
\RETURN $B$
\end{algorithmic}
\end{algorithm}

In the general case, Equation \eqref{part-sing-II} allows us to reduce the
computation of $\mathfrak{G}$ to the less involved task of computing
$\mathfrak{G}(W^{\prime})$ as soon as we detect the ordinary $i$-fold points.
To begin with treating these, here is how to find the nodes:

\begin{remark}
We know how to find \emph{all} singularities: $\operatorname{Sing}(C)$ is
given by the ideal
\[
J=\left\langle f,\frac{\partial f}{\partial X},\frac{\partial f}{\partial
Y}\right\rangle .
\]
By the Morse lemma (see \cite{Milnor}), a point $P\in\operatorname{Sing}(C)$
is a node iff the Hessian matrix $\operatorname{Hess}(f)$ formed by the second
partial derivatives of $f$ is non-degenerate at $P$. That is, $P$ is a node
iff
\[
I(P)+\left\langle \det(\operatorname{Hess}(f))\right\rangle =k{[X,Y]}\text{.}%
\]
This gives us a fast way of computing $W_{2}$.
\end{remark}

Carrying our efforts one step further, we discuss the local analysis of the
singularities via invariants. This yields an efficient method not only for
finding the delta invariant, but also for detecting the ordinary $i$-fold
points, for each $i$:

\begin{remark}
\label{rem:comp-inv-plane-curves} Let $P\in\operatorname{Sing}(C)$. After a
translation, we may assume that $P=(0,0)$ is the origin. Write $m_{P}$ for the
multiplicity and%
\[
\mu_{P}=\dim_{\;\!k}\left(  k\left[  \left[  X,Y\right]  \right]
\Big/\left\langle \frac{\partial f}{\partial X},\frac{\partial f}{\partial
Y}\right\rangle \right)
\]
for the \emph{Milnor number} of $C$ at $P$. Then $m_{P}=\deg h_{p}$, where
$h_{P}$ is the lowest degree homogeneous summand of the Taylor expansion of
$f$ at $P$. Recall that $\mu_{P}$ can be computed via standard bases (see
\cite{GP}). Furthermore, if the Newton polygon of $f$ is non-degenerate
(otherwise, successively blow up), the \emph{number of branches} of $f$ at $P$
can be computed as%
\[
r_{P}=\sum\nolimits_{j=1}^{s-1}\gcd\left(  V_{X}^{(j+1)}-V_{X}^{(j)},\text{
}V_{Y}^{(j+1)}-V_{Y}^{(j)}\right)  ,
\]
where $V^{(1)},...,V^{(s)}$ are the (ordered) vertices of the Newton polygon
(and $X$ and $Y$ refer to the respective coordinates). This is immediate from
\cite[Section 8.4, Lemma 3]{Brieskorn}. The delta invariant of $C$ at $P$ is
then obtained as
\[
\delta_{P}=\frac{1}{2}(\mu_{P}+r_{P}-1)
\]
(see, for example, \cite[Chapter 1, Proposition 3.34]{GLSsing}). Furthermore,
$P$ is an ordinary $i$-fold point iff $h_{P}$ is square-free and $m_{P}=i$.
Equivalently,
\[
{\small \left(  m_{P},r_{P},\delta_{P}\right)  =\left(  i,\;i,\;\binom{i}%
{2}\right)  }\text{.}%
\]
See \cite[Chapter 1, Proposition 3.33]{GLSsing}.

\end{remark}

The local analysis of the singularities may be used to further refine our
partition of $\operatorname{Sing}(C)$. For example, singularities of type
$ADE$ can be identified as follows:

\begin{remark}
\label{rmk charact ade}With notation as in Remark
\ref{rem:comp-inv-plane-curves}, the point $P=(0,0)\in\operatorname{Sing}(C)$ is

\begin{enumerate}
\item of type $A_{n}$, $n\geq2$, iff $h_{P}=l_{1}^{2}$, with $l_{1}\in k[X,Y]$
linear, and $\mu_{P}=n$,

\item of type $D_{n}$, $n\geq4$, iff $h_{P}=l_{1}l_{2}l_{3}$ or $h_{P}%
=l_{1}^{2}l_{2}$, with pairwise different linear polynomials $l_{j}\in
k[X,Y]$, and $\mu_{P}=n$, and

\item of type $E_{n}$, $n=6,7,8$, iff $h_{P}=l_{1}^{3}$, with $l_{1}\in
k[X,Y]$ linear, and $\mu_{P}=n$.
\end{enumerate}

Here, in (2), $h_{P}$ splits into three different linear factors iff $P$ is of
type $D_{4}$. See, for example, \cite[Chapter 1, Theorems 2.48, 2.51,
2.54]{GLSsing}.
\end{remark}

To describe the local Gorenstein adjoint ideal at a singularity of type $A$,
$D$, or $E$, we use the following notation:

\begin{notation}
For any element $g\in k[[X,Y]]$, let $g_{j}=\operatorname*{taylor}\left(
g,j\right)  \in k[X,Y] $ be the Taylor expansion of $g$ at $P=(0,0)$ modulo
$O(j+1)$.\footnote{The notation $O(m)$ stands for terms of degree $\geq m$.}
\end{notation}

If $C$ has a singularity of type $A_{n}$ at $P=(0,0)$, we may write $f$ in the
form $f=T^{2}+W^{n+1}$, where $T,W\in k[[X,Y]]$ is a regular system of
parameters.
Let $s=\left\lfloor \frac{n+1}{2}\right\rfloor $ (the meaning of $s$ will
become clear in the proof of Lemma \ref{lem Ak}). We may compute the Taylor
expansion $T_{s-1}\in k[X,Y]$ as follows. If $n$ and thus $s$ is equal to $1$,
set $T_{0}=0$. Otherwise, inductively solve $f$ for $T$: Start by choosing a
linear form $T_{1}\in k[X,Y]$ such that $\operatorname*{taylor}(f,2)=T_{1}%
^{2}$. Supposing that $1<j<s-1$ and $T_{j}=T+O(j+1)$ has already been
computed, write
\[
\operatorname*{taylor}(f-T_{j}^{2},\text{ }j+2)=2T_{1}\cdot m,
\]
with $m\in k[X,Y]$ homogeneous of degree $j+1$, and set $T_{j+1}=T_{j}+m$.

\begin{lemma}
\label{lem Ak} Let $C$ have a singularity of type $A_{n}$, $n\geq1$, at
$P=(0,0)$. Set $s=\left\lfloor \frac{n+1}{2}\right\rfloor $, and let $T_{s-1}$
be defined as above. Then $\mathfrak{G}(P)$ is the homogenization of
\[
\left\langle X^{s},\;T_{s-1},\;Y^{s}\right\rangle \subset k[X,Y]
\]
with respect to $Z$.
\end{lemma}

\begin{proof}
The case $n=1$ is clear, so we may suppose $n\geq2$. If $\mathfrak{G}^{\prime
}=\left\langle X^{s},\;T_{s-1},\;Y^{s}\right\rangle \subset k[X,Y]$, then
$\mathfrak{G}_{Q}^{\prime}=\mathcal{O}_{C,Q}$ for all $Q\in C\setminus\{P\}$,
so it suffices to show that $\mathfrak{G}_{P}^{\prime}=\mathcal{C}_{B}$, where
$B=\mathcal{O}_{C,P}$. For this, we pass to the completion
\[
\widehat{B}=k[[x,y]]=k[[X,Y]]/\langle f(X,Y)\rangle,
\]
and consider the isomorphism
\[
A=k[[t,w]]=k[[T,W]]/\left\langle T^{2}+W^{n+1}\right\rangle \rightarrow
\widehat{B}\text{, }t\mapsto T(x,y)\text{,}w\mapsto W(x,y).
\]
An analysis of the normalization algorithm applied to $A$ shows that
\[
\overline{A}=\sum_{i=0}^{n-s}k[[t]]\cdot w^{i}+\sum_{i=n-s+1}^{n}%
k[[t]]\cdot\frac{w^{i}}{t}\text{,}%
\]
and that it takes $s=\left\lfloor \frac{n+1}{2}\right\rfloor $ steps to reach
$\overline{A}$ (see \cite[Sect. 4]{BDS}). Hence,
\[
\mathcal{C}_{A}=\left\langle t,w^{s}\right\rangle _{A}\text{,}\text{ so that
}\;\mathcal{C}_{\widehat{B}}=\langle T(x,y),W(x,y)^{s}\rangle_{\widehat{B}}.
\]
Working in $k[[X,Y]]$, we write%
\[%
\begin{tabular}
[c]{lll}%
$T=aX+bY$ & and & $W=cX+dY$,
\end{tabular}
\
\]
with $a,b,c,d\in k[[X,Y]]$ and such that $ad-bc$ is a unit in $k[[X,Y]]$.
Since $\left\langle X,Y\right\rangle =\left\langle T,W\right\rangle $, it
follows that $\left\langle X,Y\right\rangle ^{s}=\left\langle T,W\right\rangle
^{s}\subset\left\langle T,W^{s}\right\rangle $. Since $\left\langle
X,Y\right\rangle =\left\langle X,T\right\rangle $ or $\left\langle
X,Y\right\rangle =\left\langle T,Y\right\rangle $, we have $W^{s}%
\in\left\langle X,Y\right\rangle ^{s}\subset\left\langle X^{s},T,Y^{s}%
\right\rangle $. We conclude that%
\[
\left\langle X^{s},T,Y^{s}\right\rangle =\left\langle T,W^{s}\right\rangle
\text{.}%
\]

If $s>1$, then $\left\langle X,Y\right\rangle =\left\langle X,T_{s-1}%
\right\rangle $ or $\left\langle X,Y\right\rangle =\left\langle T_{s-1}%
,Y\right\rangle $, hence, for any $s$, we have $\left\langle X,Y\right\rangle
^{s}\subset\left\langle X^{s},T_{s-1},Y^{s}\right\rangle $. We conclude that%
\[
\left\langle X^{s},T_{s-1},Y^{s}\right\rangle =\left\langle X^{s}%
,T,Y^{s}\right\rangle \text{.}%
\]
Now recall that $B$ is an excellent ring, which implies that $\overline
{\widehat{B}}=\widehat{\overline{B}}$ (see, for example, \cite[Sect. 1]{BDS}).
It follows that
\begin{equation}
\mathcal{C}_{\widehat{B}}=\operatorname*{Hom}\nolimits_{\widehat{B}}\left(
\overline{\widehat{B}},B\right)  =\operatorname*{Hom}\nolimits_{B}\left(
\overline{B},B\right)  \otimes_{B}\widehat{B}=\mathcal{C}_{B}\otimes
_{B}\widehat{B}\text{.} \label{equ:cond-compl}%
\end{equation}
Since completion is faithfully flat in the case considered here, we conclude
that
\[
\mathcal{C}_{B}=\left\langle x^{s},T_{s-1}(x,y),y^{s}\right\rangle
_{B}\text{.}%
\]
\end{proof}

\begin{remark}
In particular, if $P$ is a cusp, then $\mathfrak{G}(P)=\left\langle
X,\;Y\right\rangle $. So, in \eqref{part-sing-II}, nodes and cusps may be
treated simultaneously.
\end{remark}

If $C$ has a singularity of type $D_{n}$ at $P=(0,0)$, we may write $f$ in the
form $f=W\cdot\left(  T^{2}+W^{n-2}\right)  $, where $T,W\in k[[X,Y]]$ is a
regular system of parameters.
Let $s=\left\lfloor \frac{n}{2}\right\rfloor $. We may compute the Taylor
expansion $T_{s-2}\in k[X,Y]$ as follows. If $n=4$, set $T_{0}=0$. If $n\geq
5$, choose linear forms $T_{1},W_{1}\in k[X,Y]$ such that
$\operatorname*{taylor}(f,3)=T_{1}^{2}\cdot W_{1}$. For $j\leq s-2$, determine
$W_{j}=W+O(j+1)$ as the Puiseux expansion up to order $j$ of $f$ corresponding
to $W_{1}$. Supposing that $1<j<s-2$ and $T_{j}=T+O(j+1)$ has already been
computed, write
\[
\operatorname*{taylor}(f-T_{j}^{2}\cdot W_{j+1},\text{ }j+3)=2Z_{1}\cdot
W_{1}\cdot m,
\]
with $m\in k[X,Y]$ homogeneous of degree $j+1$, and set $T_{j+1}=T_{j}+m$.

\begin{lemma}
\label{lem D} Let $C$ have a singularity of type $D_{n}$, $n\geq4$, at
$P=(0,0)$. Set $s=\left\lfloor \frac{n}{2}\right\rfloor $, and let $T_{s-2}$
be defined as above. Then $\mathfrak{G}(P)$ is the homogenization of
\[
\left\langle X,\text{ }Y\right\rangle \cdot\left\langle X^{s-1},\;T_{s-2}%
,\;Y^{s-1}\right\rangle \subset k[X,Y]
\]
with respect to $Z$.
\end{lemma}

\begin{proof}
We have an isomorphism%
\[
A\rightarrow\widehat{B}\text{, }q\mapsto T(x,y)\text{, }w\mapsto
W(x,y)\text{,}%
\]
where $B=\mathcal{O}_{C,P}$ and%
\[
A=k[[t,w]]=k[[T,W]]/\left\langle W\cdot\left(  T^{2}+W^{n-2}\right)
\right\rangle .
\]
This time, the normalization is
\[
\overline{A}=\sum_{i=0}^{n-2-s}k[[t]]\cdot w^{i}+\sum_{i=n-1-s}^{n-3}%
k[[t]]\cdot\frac{w^{i}}{t}+k[[t]]\cdot\frac{w^{n-2}}{t^{2}},
\]
and it takes $s=\left\lfloor \frac{n}{2}\right\rfloor $ steps to reach
$\overline{A}$ (see again \cite[Sect. 4]{BDS}). Hence,
\[
\mathcal{C}_{A}=\left\langle t^{2},tw,w^{s}\right\rangle \text{.}%
\]

Write%
\[%
\begin{tabular}
[c]{lll}%
$T=aX+bY$ &  & $W=cX+dY$%
\end{tabular}
\ \
\]
with $a,b,c,d\in k[[X,Y]]$ and such that $ad-bc$ is a unit in $k[[X,Y]]$.
Since $\left\langle X,Y\right\rangle =\left\langle T,W\right\rangle $, we have
$\left\langle XT,YT\right\rangle =\left\langle T^{2},TW\right\rangle $ and
$\left\langle X,Y\right\rangle ^{s}=\left\langle T,W\right\rangle ^{s}%
\subset\left\langle T^{2},TW,W^{s}\right\rangle $, hence%
\[
\left\langle X,Y\right\rangle \cdot\left\langle X^{s-1},T,Y^{s-1}\right\rangle
\subset\left\langle T^{2},TW,W^{s}\right\rangle \text{.}%
\]
For the other inclusion, observe that $\left\langle X,Y\right\rangle
=\left\langle X,T\right\rangle $ or $\left\langle X,Y\right\rangle
=\left\langle T,Y\right\rangle $, so it follows that $\left\langle
X,Y\right\rangle ^{s-1}\subset\left\langle X^{s-1},T,Y^{s-1}\right\rangle $,
hence%
\[
W^{s}\in\left\langle X,Y\right\rangle ^{s}\subset\left\langle X,Y\right\rangle
\cdot\left\langle X^{s-1},T,Y^{s-1}\right\rangle \text{.}%
\]

If $s>2$, then $\left\langle X,Y\right\rangle =\left\langle X,T_{s-2}%
\right\rangle $ or $\left\langle X,Y\right\rangle =\left\langle T_{s-2}%
,Y\right\rangle $, hence, for any $s$, we have $\left\langle X,Y\right\rangle
^{s-1}\subset\left\langle X^{s-1},T_{s-2},Y^{s-1}\right\rangle $. We conclude
that%
\[
\left\langle X^{s-1},T_{s-2},Y^{s-1}\right\rangle =\left\langle X^{s-1}%
,T,Y^{s-1}\right\rangle \text{.}%
\]
To summarize,%
\[
\left\langle T^{2},TW,W^{s}\right\rangle =\left\langle X,Y\right\rangle
\cdot\left\langle X^{s-1},T,Y^{s-1}\right\rangle =\left\langle
X,Y\right\rangle \cdot\left\langle X^{s-1},T_{s-2},Y^{s-1}\right\rangle
\text{,}%
\]
hence%
\[
\mathcal{C}_{\widehat{B}}=\left\langle x,y\right\rangle \cdot\left\langle
x^{s-1},T_{s-2}(x,y),y^{s-1}\right\rangle \subset\widehat{B}\text{.}%
\]
Then the claim follows as before.
\end{proof}

\begin{lemma}
Let $C$ have a singularity of type $E_{n}$, $n=6,7,8$, at $P=(0,0)$. Set
$s=\left\lfloor \frac{n-1}{2}\right\rfloor $, and let $l_{1}$ be as in Remark
\ref{rmk charact ade}. Then $\mathfrak{G}(P)$ is the homogenization of
\[
\left\langle X,\text{ }Y\right\rangle \cdot\left\langle X^{s-1},\;l_{1}%
,\;Y^{s-1}\right\rangle \subset k[X,Y]
\]
with respect to $Z$.
\end{lemma}

\begin{proof}
Depending on $n\in\{6,7,8\}$, we have an isomorphism%
\[
A\rightarrow\widehat{B}\text{, }q\mapsto T(x,y)\text{, }w\mapsto
W(x,y)\text{,}%
\]
where $B=\mathcal{O}_{C,P}$ and%
\begin{align*}
A  &  =k[[t,w]]=k[[T,W]]/\left\langle T^{3}+W^{4}\right\rangle \text{,}\\
A  &  =k[[t,w]]=k[[T,W]]/\left\langle T\left(  T^{2}+W^{3}\right)
\right\rangle \text{,}\\
A  &  =k[[t,w]]=k[[T,W]]/\left\langle T^{3}+W^{5}\right\rangle \text{,}%
\end{align*}
respectively. In each case, by \cite[Sect. 4]{BDS},
\[
\overline{A}=k[[w]]\cdot1+k[[w]]\cdot\frac{t}{w}+k[[w]]\cdot\frac{t^{2}}%
{w^{s}}\text{,}%
\]
which implies that%
\[
\mathcal{C}_{A}=\left\langle t^{2},tw,w^{s}\right\rangle \text{.}%
\]
The same argument as in the proof of Lemma \ref{lem D} shows that
\[
\mathcal{C}_{\widehat{B}}=\left\langle x,y\right\rangle \cdot\left\langle
x^{s-1},T_{s-2}(x,y),y^{s-1}\right\rangle \subset\widehat{B}\text{,}%
\]
and the claim follows as before. Note that $T_{s-2}=0$ if $s=2$, and
$T_{s-2}=l_{1}$ if $s=3$.
\end{proof}

In principle, we could pursue a similar strategy for all singularities
classified by Arnold in \cite{Arnold}. However, in \cite{BDLS}, we give an
algorithm which, for plane curves in characteristic zero, allows us to compute
the local contributions to the normalization for a broad class of
singularities in a direct way. Combining the approach of Section
\ref{Sec General local approch} with this algorithm or with modular techniques
and normalization as described in Section \ref{sec parallel and modular}
below, we already get a very efficient algorithm for computing $\mathfrak{G}%
$.

\begin{remark}
\label{rem:alg4-and field-ext local} 
For the local analysis of the singularities, we
temporarily may have to leave $k$.
\end{remark}

\section{Parallel computation and modular
techniques\label{sec parallel and modular}}

Algorithm \ref{alg global from local} is parallel in nature since the
computations of the local adjoint ideals do not depend on each other. In this
section, in the case where the given curve is defined over $\mathbb{Q}$,
we describe a modular way of parallelizing Algorithm
\ref{alg global from local} even further. One possible approach is to replace
the computations of the Gr\"{o}bner bases involved, the computation of the
(minimal) associated primes in the singular locus, and the computations
yielding the normalizations by their modular variants as introduced by
\cite{ArnoldE}, \cite{IPS}, and \cite{BDLPSS}. These variants are either
probabilistic or require expensive tests to verify the results at the end. In
order to reduce the number and complexity of the verification tests, we
provide a direct modularization for the adjoint ideal algorithm. The approach
we propose requires only the verification of the final result: We give
efficient conditions for checking whether the result obtained is indeed the
Gorenstein adjoint ideal.

Our approach relies on the general scheme for modular computations presented
in \cite{FareyPaper}. This scheme is based on error tolerant rational
reconstruction (see Remark \ref{rmk error tolerant farey} below) and can
handle \emph{bad primes}\footnote{In our context, a prime $p$ is \emph{bad} if
Algorithm \ref{alg global from local} applied to the modulo $p$ values of the
input over the rationals does not return the reduction of the characteristic
zero result.} of various types, provided there are only finitely many such
primes. Referring to \cite{FareyPaper} for details, we will now outline the
main ideas behind the scheme.

Fix a global monomial ordering $>$ on the monoid of monomials in the variables
$X=\{X_{0},\ldots,X_{r}\}$. Consider the polynomial rings $R={\mathbb{Q}}[X]$
and, given an integer $N\geq2$, $R_{N}=(\mathbb{Z}/N\mathbb{Z})[X]$. If
$H\subset R$ or $H\subset R_{N}$ is a Gr\"obner basis, then denote by
$\operatorname{LM}(H):=\{\operatorname{LM}(f)\mid f\in H\}$ its set of leading monomials.

If $\frac{a}{b}\in\mathbb{Q}$ with $\gcd(a,b)=1$ and $\gcd(b,N)=1$, set
$\left(  \frac{a}{b}\right)  _{N}:=(a+N\mathbb{Z})(b+N\mathbb{Z})^{-1}%
\in\mathbb{Z}/N\mathbb{Z}$. If $f\in R$ is a polynomial such that $N$ is
coprime to any denominator of a coefficient of $f$, then its \emph{reduction
modulo $N$} is the polynomial $f_{N}\in R_{N}$ obtained by mapping each
coefficient $x$ of $f$ to $x_{N}$. If $H=\{h_{1},\dots,h_{t}\}\subset R$ is a
Gr\"{o}bner basis such that $N$ is coprime to any denominator in any $h_{i}$,
set $H_{N}=\{(h_{1})_{N},\dots,(h_{t})_{N}\}$. If $J\subset R$ is an ideal, we
write
\[
J_{0}=J\cap\mathbb{Z}[X]\;\text{ and }\;J_{N}=\left\langle f_{N}\mid f\in
J_{0}\right\rangle \subset R_{N}\text{,}%
\]
and call $J_{N}$ \emph{the reduction of $J$ modulo $N$}. We also write
$(R/J)_{N}=R_{N}/J_{N}$.

Based on this notation, we fix the following setup for the rest of this section:

\begin{notation}
\label{not: chapter-mod} Let $\Gamma\subset\mathbb{P}_{\mathbb{Q}}^{r}$ be a
curve of degree $n$. As before, suppose that $\Gamma$ is integral and
non-degenerate. Denote by $I(\Gamma)$ the ideal of $\Gamma$ in $R$, and by
$G(0)\subset R$ the reduced Gr\"{o}bner basis of $\mathfrak{G}(\Gamma)$. If
$p$ is a prime such that $I(\Gamma)_{p}$ is radical and defines an
integral, non-degenerate curve in $\mathbb{P}_{\mathbb{F}_{p}}^{r}$, then
write $\Gamma_{p}$ for this curve and $G(p)\subset R_{p}$ for the reduced
Gr\"{o}bner basis of $\mathfrak{G}(\Gamma_{p})$.
\end{notation}

\begin{remark}
\label{rem:reduction-practical} Given $p$, the ideal $I(\Gamma)_{p}$ can be
found using Gr\"{o}bner bases over $\mathbb Z$ (see \cite[Cor. 4.4.5]{AL} and
\cite[Lem. 6.1]{ArnoldE}). 
We will make use of this in the final verification test.
With regard to the other steps of our algorithm (in particular, in a randomized version
of the algorithm obtained by omitting the verification test), we can
proceed in the following, more efficient way: Let $\{f_{1},...,f_{r}\}$ be the
reduced Gr\"{o}bner basis of $I(\Gamma)$. Reject $p$ if one of the
$(f_{i})_{p}$ is not defined (there are only finitely many such primes $p$).
Otherwise, realize $I(\Gamma)_{p}$ via the equality
\begin{equation}
I(\Gamma)_{p}=\left\langle (f_{1})_{p},...,(f_{r})_{p}\right\rangle \subset
R_{p}, \label{eq:reduction-practical}%
\end{equation}
which holds true for all but finitely many primes $p$. These finitely many bad
primes will not influence the lift if we apply error tolerant rational
reconstruction as described in Remark \ref{rmk error tolerant farey} below.
\end{remark}

\begin{remark}
\label{rem:reduction-prop} There are only finitely many primes $p$ for which
the desired conditions on $I(\Gamma)_{p}$ in Notation \ref{not: chapter-mod}
are not satisfied. Since these conditions can be checked using polynomial
factorization and Gr\"{o}bner bases, we may simply reject such a bad prime if
we encounter it in our modular algorithm. Hence, we will ignore these bad
primes in the following discussion. In particular, we will assume that the
Gr\"{o}bner bases $G(p)$ exists for all primes $p$.
\end{remark}

The basic idea of the modular adjoint ideal algorithm can then be described as
follows: First, choose a set of primes $\mathcal{P}$ and compute $G(p)$ for
each $p\in\mathcal{P}$. Second, lift the $G(p)$ coefficientwise to a set of
polynomials $G\subset R$. Provided that $\mathfrak{G}(\Gamma)_{p}%
=\mathfrak{G}(\Gamma_{p})$ for each $p\in\mathcal{P}$, we then expect that $G$
is a Gr\"{o}bner basis which coincides with our target Gr\"{o}bner basis
$G(0)$.

The lifting process consists of two steps. First, use Chinese remaindering to
lift the $G(p)\subset R_{p}$ to a set of polynomials $G(N)\subset R_{N}$, with
$N:=\prod_{p\in\mathcal{P}}p$\ . Second, compute a set of polynomials
$G\subset R$ by lifting the coefficients occurring in $G(N)$ to rational
coefficients. Here, to identify Gr\"obner basis elements corresponding to each
other, we require that $\operatorname{LM}(G(p)) =\operatorname{LM}(G(q))$ for
all $p,q\in\mathcal{P}$. This leads to condition (L2) in the definition below:

\begin{definition}
\label{defnLucky}With notation as above, a prime $p$ is called \emph{lucky} if:

\begin{enumerate}
\item[(L1)] $\mathfrak{G}(\Gamma)_{p}=\mathfrak{G}(\Gamma_{p})$ and

\item[(L2)] $\operatorname{LM}(G(0))=\operatorname{LM}(G(p))$.
\end{enumerate}

Otherwise $p$ is called \emph{unlucky}.
\end{definition}

\begin{lemma}
\label{lem:lucky-primes-finite} All but finitely many primes are lucky.
\end{lemma}

\begin{proof}
As is clear from the proof of \cite[Lemma 5.5]{FareyPaper}, it is enough to
show that condition (L1) is true for all but finitely many primes.
For this, we may assume that both $\Gamma$ and $\Gamma_{p}$ do not have any
singularities at $X_{0}=0$. Let $C$ be the affine part of $\Gamma$. Write
$A={\mathbb{Q}}[X_{1},...,X_{r}]/I(C)$. As shown in \cite{BDLPSS},
$(\overline{A})_{p}=\overline{A_{p}}$ for all but finitely many primes $p$. So
if we write $\overline{A}=\frac{1}{d}U$, with an ideal $U\subset A$ and an
element $0\neq d\in A$, and $\overline{A_{p}}=\frac{1}{d(p)} U(p)$, with
$U(p)\subset A_{p}$ and $d_{p}\in A_{p}$, then%
\[
(d_{p}:U_{p})=(d(p):U(p))
\]
for all but finitely many primes $p$. Computing an ideal quotient amounts to a
Gr\"{o}bner basis computation. Hence, as pointed out in \cite[Remark
5.3]{FareyPaper},
\[
(d:U)_{p}=(d_{p}:U_{p})
\]
for all but finitely many primes $p$. The result follows, thus, from
Propositions \ref{Cor adjoint ideal from conductor} and \ref{prop adjoint ideal and conductor}.
\end{proof}

When performing our modular algorithm, condition (L1) can only be checked a
posteriori: We compute $G(p)$ and, thus, $\mathfrak{G}(\Gamma_{p})$ on our
way, but $\mathfrak{G}(\Gamma)_{p}$ is only known to us after $G(0)$ and,
thus, $\mathfrak{G}(\Gamma)$ has been computed. This is not a problem,
however, since the finitely many primes where $\mathfrak{G}(\Gamma)_{p}%
\not =\mathfrak{G}(\Gamma_{p})$ will not influence the final result if we
apply error tolerant rational reconstruction and the set $\mathcal{P}$ is
large enough:

\begin{remark}
\label{rmk error tolerant farey} Let $N^{\prime}$ and $M$ be integers with
$\gcd(N^{\prime},M)=1$, let $N=N^{\prime} \cdot M$, and let $\frac{a}{b}%
\in\mathbb{Q}$ with $\gcd(a,b)=\gcd(N^{\prime},b)=1$. Set $r_{1}:=\left(
\frac{a}{b}\right)  _{N^{\prime}}\in\mathbb{Z}/N^{\prime}\mathbb{Z}$, let
$r_{2}\in\mathbb{Z}/M\mathbb{Z}$ be arbitrary, and denote by $r$ the image of
$(r_{1},r_{2})$ under the isomorphism%
\[
\mathbb{Z}/N^{\prime}\mathbb{Z}\times\mathbb{Z}/M\mathbb{Z}\rightarrow
\mathbb{Z}/N\mathbb{Z}\text{.}%
\]
Lifting $r$ to a rational number by Gaussian reduction, starting from
$(a_{0},b_{0})=(N^{\prime}M,0)$ and $(a_{1},b_{1})=(r,1)$, we create the
sequence $(a_{i},b_{i})$ obtained by
\[
(a_{i+2},b_{i+2})=(a_{i},b_{i})-q_{i}(a_{i+1},b_{i+1})\text{,}%
\]
with
\[
q_{i}=\left\lfloor \frac{\langle(a_{i},b_{i}), (a_{i+1},b_{i+1})\rangle}%
{\Vert(a_{i+1}, b_{i+1}) \Vert^{2}} \right\rceil .
\]
Computing this sequence until ${\Vert(a_{i+2}, b_{i+2})\Vert}\geq
{\Vert(a_{i+1}, b_{i+1}) \Vert}$, we return \texttt{false} if ${\Vert(a_{i+1},
b_{i+1})\Vert^{2}} \geq N$, and $\frac{a_{i+1}}{b_{+1}}$, otherwise. By
\cite[Lemma 4.3]{FareyPaper}, this algorithm will return $\frac{a_{i+1}%
}{b_{+1}}=\frac{a}{b}$, provided that $N$ is large enough and $M\ll N^{\prime
}$. More precisely, we ask that $N^{\prime}>(a^{2}+b^{2})\cdot M$.
\end{remark}

\begin{definition}
If $\mathcal{P}$ is a finite set of primes, set
\[
N^{\prime}=\prod_{p\in\mathcal{P}\text{ lucky}}p\hspace{0.5cm}\text{and}%
\hspace{0.5cm}M=\prod_{p\in\mathcal{P}\text{ unlucky}}p\text{.}%
\]
Then $\mathcal{P}$ is called \emph{sufficiently large} if
\[
N^{\prime}>(a^{2}+b^{2})\cdot M
\]
for all coefficients $\frac{a}{b}$ of polynomials in $G(0)$ (assume
$\mathop{gcd}(a,b)=1$).
\end{definition}

\begin{lemma}
\label{lem sufficiently general} If $\mathcal{P}$ is a sufficiently large set
of primes satisfying condition (L2), then the reduced Gr\"{o}bner bases
$G(p)$, $p\in\mathcal{P}$, lift to the reduced Gr\"{o}bner basis $G(0)$.
\end{lemma}

\begin{proof}
See \cite[Lemma 5.6]{FareyPaper}.
\end{proof}

From a theoretical point of view, Lemma \ref{lem:lucky-primes-finite}
guarantees that a sufficiently large set $\mathcal{P}$ of primes satisfying
condition (L2) exists. From a practical point of view, however, (L2) can only
be checked a posteriori. 
Nevertheless, in order to be able to identify Gr\"{o}bner basis elements
in the lifting process, we have to restrict to a set of primes $p$ which all
have the same associated set of lead monomials $\operatorname{LM}(G(p))$. 
Hence, taking Lemma \ref{lem:lucky-primes-finite} into account, we proceed along the
following lines: First, fix an integer $t\geq1$ and choose a set of $t$ primes
$\mathcal{P}$ at random. Second, compute $\mathcal{GP}=\{G(p)\mid
p\in\mathcal{P}\}$ and use a majority vote with respect to (L2):\vspace{0.2cm}

\emph{\textsc{deleteByMajorityVote:} Define an equivalence relation on
$\mathcal{P}$ by setting $p\sim q:\Longleftrightarrow\operatorname{LM}%
(G(p))=\operatorname{LM}(G(q)).$ Then replace $\mathcal{P}$ by the equivalence
class of largest cardinality,\footnote{We have to use a weighted cardinality
count: when enlarging $\mathcal{P}$, the total weight of the elements already
present must be strictly smaller than the total weight of the new elements.
Otherwise, though highly unlikely in practical terms, it may happen that only
unlucky primes are accumulated.} and change $\mathcal{GP}$ accordingly.}%
\vspace{0.2cm}

\noindent Now, all $G(p)$, $p\in\mathcal{P}$, have the same set of leading
monomials. Hence, we can apply the rational reconstruction algorithm to the
coefficients of the Gr\"{o}bner bases in $\mathcal{GP}$. If this algorithm
returns \texttt{false} at some point, we enlarge the set $\mathcal{P}$ by $t$
primes not used so far, and repeat the whole process. Otherwise, the lifting
yields a set of polynomials $G\subset R$. Furthermore, if $\mathcal{P}$ is
sufficiently large, all primes in $\mathcal{P}$ satisfy condition (L2). Since
we cannot check, however, whether $\mathcal{P}$ is sufficiently large, a final
verification step is needed. Since this may be expensive, especially if $G\neq
G(0)$, we first perform a test in positive characteristic:

\emph{\textsc{pTest:} Randomly choose a prime $p\notin\mathcal{P}$ which does
not divide the numerator or denominator of any coefficient occurring in a
polynomial in $G$. Return \texttt{true} if $G_{p}=G(p)$, and \texttt{false}
otherwise.}

\vspace{0.2cm}

If \textsc{pTest} returns \texttt{false}, then $\mathcal{P}$ is not
sufficiently large (or the extra prime chosen in \textsc{pTest} is bad). In
this case, we enlarge $\mathcal{P}$ as above and repeat the process. If
\textsc{pTest} returns \texttt{true}, however, then most likely $G=G(0)$. In
this case, we verify the result over the rationals as described below. If the
verification fails, we again enlarge $\mathcal{P}$ and repeat the process.

We now discuss the verification. We write $I=\langle G\rangle_{R}$ for the
lifted modular result and $\mathfrak{G}=\mathfrak{G}(\Gamma)\subset R$ for the
correct result. After checking that $G$ is indeed a Gr\"{o}bner basis
and $I$ is saturated (henceforth, this will be assumed), we apply the following results.

\begin{lemma}
\label{lem verif} With notation as above, the ideal $I$ is equal to the
Gorenstein adjoint ideal $\mathfrak{G}$ of $\Gamma$ iff

\begin{enumerate}
\item $I(\Gamma)\subsetneqq I$,

\item $\deg\Delta(I)=\deg I+\delta(\Gamma)$, and

\item $\deg I=\deg\mathfrak{G}$.
\end{enumerate}
\end{lemma}

\begin{proof}
If $I=\mathfrak{G}$, then $I$ satisfies (1), (2), and (3). Conversely, by
Lemma \ref{lem deg double point divisor}, conditions $(1)$ and $(2)$ imply
that $I$ is an adjoint ideal of $\Gamma$. In this case, since $\mathfrak{G}$
is the largest such ideal, we have $I\subset\mathfrak{G}$. But then
$I=\mathfrak{G}$ by (3).
\end{proof}

It is clear how to check condition (1). In what follows, we describe a method
for checking (2) which, in particular, provides a way of finding $\deg
\Delta(\mathfrak{G})$. This will allow us to check (3) via the formula
$\deg\mathfrak{G}=\deg\Delta(\mathfrak{G})-\delta(\Gamma)$.

If ${k}$ is any field, and $A$ is any reduced Noetherian ${k}%
$-algebra, the \emph{delta invariant} of $A$ is defined to be
\[
\delta_{{k}}(A)=\dim_{{k}} \overline{A}/A.
\]

\begin{proposition}
\label{thm delta halbstetig} Let $B$ be a ring, and let $A$ be a $B$-algebra
with the following properties:

\begin{enumerate}
\item $\left(  B, \mathfrak{m}\right)  $ is a normal local ring with perfect
residue class field ${k}$.

\item $B\rightarrow\widehat{B}$ is flat, and for all $\mathfrak{p}%
\in\operatorname{Spec}(B)$ such that $\mathfrak{p}\widehat{B}\neq\widehat{B}$,
the ring $\widehat{B}\otimes B_{\mathfrak{p}}/\mathfrak{p }B_{\mathfrak{p}}$
is geometrically normal.

\item $A$ is a formally equidimensional Nagata ring.

\item $A$ is a flat $B$-algebra, $\mathfrak{m}A$ is contained in every maximal
ideal of $A$, $A/\mathfrak{m}A$ is reduced, and $\delta_{{k}%
}(A/\mathfrak{m}A)<\infty$.

\item $\overline{A}/A$ is a finite $B$-module.

\item The unique map $\overline{A}/\mathfrak{m }\overline{A}\rightarrow
\overline{A/\mathfrak{m}A}$ factorizing the normalization map ${A}/\mathfrak{m
}{A}\rightarrow\overline{A/\mathfrak{m}A}$ as
\[
{A}/\mathfrak{m }{A}\rightarrow\overline{A}/\mathfrak{m }\overline
{A}\rightarrow\overline{A/\mathfrak{m}A}
\]
is injective.
\end{enumerate}

Then
\[
\delta_{\operatorname{Q}(B)}(A\otimes_{B}\operatorname{Q}(B))\leq
\delta_{{k}}(A/\mathfrak{m}A).
\]

\end{proposition}

\begin{proof}
See \cite[Prop.~2.2.1(i)]{Lipman} for the factorization in (6) and
\cite[Prop.~3.3]{Lipman} for the proof of the proposition.
\end{proof}

\begin{corollary}
\label{cor delta halbstetig} In the setting of Notation \ref{not: chapter-mod}%
, given a prime $p$, we have
\[
\delta(\Gamma)\leq\delta(\Gamma_{p}).
\]

\end{corollary}

\begin{proof}
Let $X^{\prime}=\{X_{1},\ldots,X_{r}\}$. We may assume that $\Gamma$
has no singularities at $X_{0}=0$. As before, let $C$ be the
affine part of $\Gamma$. Then $J:=I(C)_{0}\subset\mathbb{Z}[X^{\prime}]$ is a
prime ideal of height $n-1$, $\left\langle p,J\right\rangle $ is a prime
ideal, and $J\cap\mathbb{Z}=\left\langle 0\right\rangle $. The claim follows
by applying Proposition \ref{thm delta halbstetig} to $(B,\mathfrak{m}%
)=(\mathbb{Z}_{\left\langle p\right\rangle },\left\langle p\right\rangle )$
and $A=\mathbb{Z}_{\left\langle p\right\rangle }[X^{\prime}]/J\;\!\mathbb{Z}%
_{\left\langle p\right\rangle }[X^{\prime}]$ since, then, $A\otimes
_{B}\operatorname{Q}(B)=\mathbb{Q}[X^{\prime}]/I(C)$ and $A/\mathfrak{m}%
A=\mathbb{F}_{p}[X^{\prime}]/I(C)_{p}$, and conditions $(1)$ through $(6)$ of
the proposition are satisfied. Indeed, this is clear for $(1)$, and $(2)$ holds
since $B$ is excellent. Moreover, we have $(3)$ since $A$ is of finite type
over $B$ and $J\;\!\mathbb{Z}_{\left\langle p\right\rangle }[X^{\prime}]$ is a
prime ideal. Condition $(4)$ follows since $A$ is a torsion free $B$-module,
$\left\langle p,J\right\rangle $ is a prime ideal, and $\operatorname{Spec}%
(A/\mathfrak{m}A)$ is a curve. We obtain $(5)$ since $A/\mathcal{C}_{A}$ is a
finite $B$-module and $\overline{A}/\mathcal{C}_{A}$ is a finite
$A/\mathcal{C}_{A}$-module. Condition $(6)$ follows from Lemma
\ref{lem delta halbstetig} below which gives us a canonical map
\[
\overline{A}\rightarrow\overline{A/mA}\text{, }\alpha=\frac{\overline{a}%
}{\overline{b}}\mapsto\frac{a\operatorname{mod}\left\langle p,J\right\rangle
}{b\operatorname{mod}\left\langle p,J\right\rangle },
\]
where $\overline{a},\overline{b}$ are the images of $a,b\in\mathbb{Z}%
_{\left\langle p\right\rangle }[X^{\prime}]$ in $A$, and $b\notin\left\langle
p,J\right\rangle $. Since $\alpha=\frac{\overline{a}}{\overline{b}}$ is in the
kernel of this map iff $a\in\left\langle p,J\right\rangle $, we get an
injective map $\overline{A}/\mathfrak{m}\overline{A}\rightarrow\overline
{A/\mathfrak{m}A}$ which factors the normalization map as desired.
\end{proof}

Before deriving Lemma \ref{lem delta halbstetig}, we illustrate condition
$(6)$ by an example.

\begin{example}
Let $(B, \mathfrak{m})=(\mathbb{Z}_{\left\langle 3\right\rangle },
\left\langle 3\right\rangle )$ and $A=\mathbb{Z}_{\left\langle 3\right\rangle
}[X,Y]/\left\langle X^{3}+Y^{3}+Y^{5}\right\rangle $. Then $\overline
{A/\mathfrak{m}A}=\left\langle 1,\frac{x}{y},\frac{(x+y)^{2}}{y^{3}%
}\right\rangle _{A/\mathfrak{m}A}$ and $\overline{A}=\left\langle 1,\frac
{x}{y},\frac{x^{2}}{y^{2}}\right\rangle _{A}$. We compute $\delta_{\mathbb{Q}%
}(A\otimes_{B}\mathbb{Q})=3$ and $\delta_{\mathbb{F}_{p}}(A/\mathfrak{m}A)=4$,
and find that
\[
\overline{A}/\mathfrak{m}\overline{A}=\left\langle 1,\frac{x}{y},\frac{x^{2}%
}{y^{2}}\right\rangle _{A/\mathfrak{m}A}\subsetneqq\left\langle 1,\frac{x}%
{y},\frac{(x+y)^{2}}{y^{3}}\right\rangle _{A/\mathfrak{m}A}=\overline
{A/\mathfrak{m}A}\text{.}%
\]

\end{example}

\begin{lemma}
\label{lem delta halbstetig}With the notation of the proof of Corollary
\ref{cor delta halbstetig}, for any $\alpha\in\overline{A}$ there exist
$a,b\in\mathbb{Z}[X^{\prime}]$ with $b\notin\left\langle p,J\right\rangle $
such that $\alpha=\frac{\overline{a}}{\overline{b}}$.
\end{lemma}

\begin{proof}
For $\alpha\in\overline{A}\subset\operatorname{Q}(A)=\operatorname{Q}%
(\mathbb{Z}[X^{\prime}]/J)$, there are $a,b\in\mathbb{Z}[X^{\prime}]$ with
$b\notin J$ and $\alpha=\frac{a\operatorname{mod}J}{b\operatorname{mod}J}$,
and there are $a_{0},\ldots,a_{m-1}\in\mathbb{Z}[X^{\prime}]$ and
$d\in\mathbb{Z}$ with $p\nmid d$ such that%
\[
\alpha^{m}+\frac{a_{m-1}\operatorname{mod}J}{d}\alpha^{m-1}+\ldots+\frac
{a_{0}\operatorname{mod}J}{d}=0\text{,}%
\]
that is, $d\cdot a^{m}+a_{m-1}\cdot ba^{m-1}+\ldots+a_{0}\cdot b^{m}\in J$.

If $b\in\left\langle p,J\right\rangle $, then $d\cdot a^{m}\in\left\langle
p,J\right\rangle $, hence, since $J$ is radical, $a\in\left\langle
p,J\right\rangle $. Then $a=pa_{1}+c_{1}$ and $b=pb_{1}+d_{1}$ with
$a_{1},b_{1}\in\mathbb{Z}[X^{\prime}]$ and $c_{1},d_{1}\in J$. If $b_{1}%
\in\left\langle p,J\right\rangle $, we can iterate the process. Inductively,
we obtain $a_{s},b_{s}\in\mathbb{Z}[X^{\prime}]$ and $c_{s},d_{s}\in J$ with
$a=p^{s}a_{s}+c_{s}$ and $b=p^{s}b_{s}+d_{s}$. If $b_{s}\in\left\langle
p,J\right\rangle $ for all $s$, then $b\in%
{\displaystyle\bigcap\nolimits_{s}}
\left\langle p^{s},J\right\rangle =J$, a contradiction. Otherwise there is an
$s$ with $b_{s}\notin\left\langle p,J\right\rangle $. Then%
\[
\alpha=\frac{a\operatorname{mod}J}{b\operatorname{mod}J}=\frac{p^{s}%
a_{s}\operatorname{mod}J}{p^{s}b_{s}\operatorname{mod}J}=\frac{a_{s}%
\operatorname{mod}J}{b_{s}\operatorname{mod}J}\text{.}%
\]

\end{proof}

In the following, we write again $\pi:\overline{\Gamma}\rightarrow\Gamma$ for
the normalization map, and denote by $M$ the vanishing ideal of
$\operatorname{Sing}(\Gamma)$ in $R$. Consider a homogeneous polynomial $g\in
I=\left\langle G\right\rangle _{R}$ not contained in $I(\Gamma)$, and let $m$
be its degree. Let $\operatorname{div}(g)$ be the divisor cut out by
$\pi^{\ast}g$ on $\overline{\Gamma}$, let $D(g)=\operatorname{div}%
(g)-\Delta(I)$ be the corresponding divisor in $\left\vert mH-\Delta
(I)\right\vert $, and let $d(g)=\deg D(g)$. Furthermore, write $\widetilde
{d}(g)$ for the degree of the part of $D(g)$ away from $\operatorname{Sing}%
(\Gamma)$. Then $\widetilde{d}(g)\leq d(g)$, and $\widetilde{d}(g)$ can be
computed as%
\[
\widetilde{d}(g)=\deg\left(  (I(\Gamma)+\left\langle g\right\rangle
):M^{\infty}\right)
\]
provided that $I:M^{\infty}=\left\langle 1\right\rangle $, what we will
henceforth assume (in Algorithm \ref{algModNormal} below, if this condition is
not fulfilled, we enlarge our set of primes).

\begin{theorem}
Let $I=\langle G\rangle_{R}$ be as above, and let $p$ be a prime number. Suppose:

\begin{enumerate}
\item $\operatorname{LM}(I(\Gamma_{p}))=\operatorname{LM}(I(\Gamma))$,

\item $G(p)$ is a Gr\"{o}bner basis of an adjoint ideal of $\Gamma_{p}$,

\item $G_{p}=G(p)$,

\item $\widetilde{d}(g_{p})=\deg(\Gamma)\cdot m-\deg(\left\langle
G(p)\right\rangle _{R_{p}})-\delta(\Gamma)$, and

\item $m$ is large enough to ensure that $\left\vert m H-\Delta(I)\right\vert
$ is nonspecial.
\end{enumerate}
Then%
\[
\deg\Delta(I)=\deg(\Gamma)\cdot m-\widetilde{d}(g_{p})\text{.}%
\]
Furthermore, $\deg\Delta(I)=\deg\Delta(I_{p})$, and $I$ is an adjoint ideal of
$\Gamma$.
\end{theorem}

\begin{remark}
To apply the theorem in the setup above, note: Condition $(1)$ can easily be
tested. Furthermore, $(2)$ and $(3)$ are satisfied by the construction of $G$.
Since we know how to compute $\delta(\Gamma)$, condition $(4)$ can be tested.
With respect to $(5)$, we will comment on how to choose $m$ in Lemma
\ref{lem bound m} below.
\end{remark}

\begin{proof}
[Proof of the theorem]By $(1)$, $\deg(\Gamma_{p})=\deg(\Gamma)$ and
$p_{a}(\Gamma_{p})=p_{a}(\Gamma)$. First note, that by $(3)$%
\[
I_{p}=\left\langle G\right\rangle _{R_{p}}=\left\langle G(p)\right\rangle
_{R_{p}}%
\]
and, as $G$ is assumed to be a Gr\"{o}bner basis,%
\begin{equation}
\deg(\left\langle G\right\rangle _{R})=\deg(\left\langle G(p)\right\rangle
_{R_{p}})\text{.} \label{equ deg G Gp}%
\end{equation}
By Corollary \ref{cor delta halbstetig}, we have $\delta(\Gamma)\leq
\delta(\Gamma_{p})$. Hence%
\begin{align*}
\widetilde{d}(g_{p})  &  \leq d(g_{p})=\deg(\Gamma_{p})\cdot m-\deg
\Delta(I_{p})\\
&  =\deg(\Gamma)\cdot m-\deg(I_{p})-\delta(\Gamma_{p})\\
&  \leq\deg(\Gamma)\cdot m-\deg(I_{p})-\delta(\Gamma)
\end{align*}
using that by $(2)$ the ideal $I_{p}$ is an adjoint ideal of $\Gamma_{p}$. By
$(4)$ the chain of inequalities is an equality, hence%
\[
\widetilde{d}(g_{p})=d(g_{p})=\deg(\Gamma_{p})\cdot m-\deg\Delta(I_{p})
\]
and%
\[
\delta(\Gamma)=\delta(\Gamma_{p})\text{.}%
\]
By (\ref{equ deg G Gp}) and Lemma \ref{lem deg double point divisor} this
implies that%
\begin{equation}
\deg\Delta(I_{p})=\deg(I_{p})+\delta(\Gamma_{p})=\deg(I)+\delta(\Gamma
)\geq\deg\Delta(I)\text{,} \label{equ deg delta I}%
\end{equation}
or equivalently%
\[
d(g_{p})\leq d(g)\text{.}%
\]

To prove equality, we consider the closed subscheme%
\[%
\mathcal{X}=V(I(\Gamma)_{0})\subset \mathbb{P}_{\mathbb{Z}}^{r} \overset{\pi}{\longrightarrow} \operatorname{Spec}\mathbb{Z}
\]
with projection $\pi$ and fibers $\mathcal{X}_{q}=\pi^{-1}(\left\langle
q\right\rangle )=\mathcal{X}\times_{\operatorname{Spec}\mathbb{Z}%
}\operatorname{Spec}\kappa(\left\langle q\right\rangle )$. So over the generic
point $\left\langle 0\right\rangle \in\operatorname{Spec}\mathbb{Z}$ the fiber
is $\mathcal{X}_{0}=\Gamma$ and over $\left\langle p\right\rangle $ it is
$\mathcal{X}_{p}=\Gamma_{p}$. By $(1)$ the Hilbert polynomials of $\Gamma$ and
$\Gamma_{p}$ are equal, hence there is a Zariski open subset $V\subset
\operatorname{Spec}\mathbb{Z}$ with $\left\langle p\right\rangle \in V$ such
that the Hilbert polynomial is constant on $V$. So $\pi_{V}:\mathcal{X}%
_{V}=\pi^{-1}(V)\rightarrow V$ is a flat family (see \cite[Ch. III, Thm.
9.9]{Hartshorne}).

Since $\delta(\Gamma_{p})=\delta(\Gamma)$, the $\delta$-constant criterion for
simultaneous normalization (see \cite{Lipman}) implies that there is a Zariski
open subset $U\subset V\subset\operatorname{Spec}\mathbb{Z}$ with
$\left\langle p\right\rangle \in U$ such that $\pi_{U}$ is equinormalizable.
That is, there is a finite map $\nu:Z\rightarrow\mathcal{X}_{U}$ such
that $\overline{\pi}:=\pi_{U}\circ\nu$ is flat with nonempty geometrically
normal fibers, and for each $\left\langle q\right\rangle \in U$ the induced
map on the fibers $\nu_{q}:\overline{\mathcal{X}}_{q}=\overline{\pi}%
^{-1}(\left\langle q\right\rangle )\rightarrow\pi^{-1}(\left\langle
q\right\rangle )=\mathcal{X}_{q}$ is a normalization map.

Since, by construction, the family of sheaves defined by $I_0$ is flat over $U$ and $U$ contains both
$\left\langle 0\right\rangle $ and $\left\langle p\right\rangle $, the
semicontinuity theorem (see, for example, \cite[Ch. 5, Thm. 3.20]{Liu})
implies that the dimensions of the linear series induced by $I$ on
$\overline{\Gamma}$ and by $I_{p}$ on $\overline{\Gamma}_{p}$ satisfy%
\[
h^{0}\left(  \overline{\Gamma}_{p},\mathcal{O}_{\overline{\Gamma}_{p}}(m\cdot
H_{p}-\Delta(I_{p}))\right)  \geq h^{0}\left(  \overline{\Gamma}%
,\mathcal{O}_{\overline{\Gamma}}(m\cdot H-\Delta(I))\right)  \text{.}%
\]
Hence by $(5)$, Riemann-Roch, and $\delta(\Gamma_{p})=\delta(\Gamma)$ it
follows that the degrees of the linear series satisfy $d(g_{p})\geq d(g)$, so
we obtain the second equality in%
\[
\widetilde{d}(g_{p})=d(g_{p})=d(g)%
\]
(having shown the first already above). The second equality also translates
into $\deg\Delta(I_{p})=\deg\Delta(I)$ which, by \eqref{equ deg delta I},
implies that $I$ is an adjoint ideal. Moreover,%
\[
\deg(\Gamma)\cdot m-\deg\Delta(I)=\deg(\Gamma_{p})\cdot m-\deg\Delta
(I_{p})=\widetilde{d}(g_{p})\text{.}%
\]

\end{proof}

\begin{remark}
Suppose now, in addition to the previous assumptions, that $I_{p}$ is the
Gorenstein adjoint ideal of $\Gamma_{p}$. Since $I$ is an adjoint ideal of
$\Gamma$, we have $I\subset\mathfrak{G}$ which implies $\deg I\geq
\deg\mathfrak{G}$, hence%
\begin{equation}
\deg\Delta(\mathfrak{G})=\deg(\mathfrak{G})+\delta(\Gamma)\leq\deg
(I)+\delta(\Gamma)=\deg\Delta(I)=\deg\Delta(I_{p})\text{.}
\label{equ Gorenstein inequality}%
\end{equation}
Moreover, by semicontinuity
\[
\dim\left\vert m\cdot H_{p}-\Delta(I_{p})\right\vert \geq\dim\left\vert m\cdot
H-\Delta(\mathfrak{G})\right\vert
\]
for $m$ large enough, so by Riemann-Roch and $\delta(\Gamma_{p})=\delta
(\Gamma)$ we have%
\[
\deg\Delta(I_{p})\leq\deg\Delta(\mathfrak{G})\text{.}%
\]
Hence \eqref{equ Gorenstein inequality} is an equality and implies%
\[
\deg I=\deg\mathfrak{G}\text{,}%
\]
that is, $I$ is the Gorenstein adjoint ideal of $\Gamma$.
\end{remark}

In order to expect condition $(4)$ to be satisfied for randomly chosen $g$ and
$p$, the degree $m$ has to be chosen large enough such that $\widetilde
{d}(g)=\deg(\Gamma)\cdot m-\deg\Delta(\mathfrak{G})$ for a generic
$g\in\mathfrak{G}_{m}$ (taking into account that $\widetilde{d}(g)=\widetilde
{d}(g_{p})$, and $\delta(\Gamma)=\delta(\Gamma_{p})$ and $I_{p}=\mathfrak{G}%
_{p}$, hence $\deg\Delta(I_{p})=\deg\Delta(\mathfrak{G})$ holds true for all
but finitely many primes $p$). The following lemma specifies an appropriate
bound for $m$, which will also be sufficient to obtain $(5)$.

\begin{lemma}
\label{lem bound m}Consider an integer $m$ such that $\operatorname*{P}%
\nolimits_{\Gamma}(m)-1\geq p_{a}(\Gamma)$ and suppose that $g\in
\mathfrak{G}_{m}$ is generic. Then%
\[
\widetilde{d}(g)=\deg(\Gamma)\cdot m-\deg\Delta(\mathfrak{G})
\]
Furthermore, $\left\vert mH-\Delta(\mathfrak{G})\right\vert $ is nonspecial.
\end{lemma}

\begin{proof}
By assumption and since $\operatorname*{P}\nolimits_{\Gamma}(m)=(\deg
\Gamma)\cdot m-p_{a}(\Gamma)+1$, we have%
\[
\deg(\Gamma)\cdot m\geq2p_{a}(\Gamma).
\]
By Corollary \ref{cor delta gorenstein adjoint ideal}, we obtain $\deg
\Delta(\mathfrak{G})\leq2\delta(\Gamma)$. Hence, it follows that%
\begin{align*}
\deg(\Gamma)\cdot m-\deg\Delta(\mathfrak{G}) &  \geq\deg(\Gamma)\cdot
m-2\delta(\Gamma)\\
&  =\deg(\Gamma)\cdot m-2p_{a}(\Gamma)+2p(\Gamma)\geq2p(\Gamma).
\end{align*}
This implies that $\left\vert mH-\Delta(\mathfrak{G})\right\vert $ is
base-point free (see \cite[Ch. IV, Cor. 3.2]{Hartshorne}), hence, since $g$ is
generic, we have $d(g)=\widetilde{d}(g)$. By reason of its degree, the linear
series is also nonspecial (see \cite[Ch. IV, Ex. 1.3.4]{Hartshorne}).
\end{proof}

\begin{remark}
For a plane curve $\Gamma$ of degree $n$ the condition $\operatorname*{P}%
\nolimits_{\Gamma}(m)-1\geq p_{a}(\Gamma)$ is equivalent to $n\cdot
m\geq(n-1)(n-2)$, which is satisfied for $m\geq n-2$.
\end{remark}

We summarize our approach in Algorithm \ref{algModNormal}.

\begin{algorithm}[H]
\caption{Modular adjoint ideal}
\label{algModNormal}
\begin{algorithmic}[1]
\REQUIRE A curve $\Gamma\subset\mathbb{P}^{r}$ satisfying the conditions of
Notation \ref{not: chapter-mod}.
\ENSURE The Gorenstein adjoint ideal $\mathfrak{G}(\Gamma)$.
\STATE choose an integer $t\geq 1$
\STATE $\mathcal{P}=\mathcal{GP}=\emptyset$
\LOOP
\STATE\label{state:step 4} choose a list $\mathcal{Q}$ of $t$ random primes not used so far
\FORALL {$p\in\mathcal{Q}$}
\IF {$\Gamma_{p}$ is irreducible, non-degenerate, \textbf{and} $\operatorname{LM}(I(\Gamma)) = \operatorname{LM}(I(\Gamma_{p}))$}
\STATE compute the reduced Gr\"{o}bner basis $G_p$ of $\mathfrak{G}(\Gamma_p)\subset R_p$
(via Alg. \ref{alg global from local})
\STATE $\mathcal{P}=\mathcal{P}\cup\{p\}$, $\mathcal{GP}=\mathcal{GP}\cup\{G_p\}$
\ENDIF
\ENDFOR
\STATE $(\mathcal{GP},\mathcal{P})=\textsc{deleteByMajorityVote}(\mathcal{GP}%
,\mathcal{P})$
\STATE  lift $(\mathcal{GP},\mathcal{P})$ to a set of polynomials $G\subset R$ via the Chinese
remainder theorem and Gaussian reduction
\IF {the lifting succeeds \textbf{and} \textsc{pTest}$(I(\Gamma),G,\mathcal{P})$}
\IF {$G$ is a Gr\"obner basis \textbf{and} $\left\langle G\right\rangle$ is saturated \textbf{and} $\left\langle G\right\rangle :M^{\infty}=\left\langle 1\right\rangle $}
\STATE choose $m$ such that $\operatorname*{P}\nolimits_{\Gamma}(m)-1\geq p_{a}(\Gamma)$
\STATE choose $g\in\left\langle G\right\rangle _{m}$ at random
\STATE choose a prime $p\in \mathcal{P}$
\STATE compute $\widetilde{d}(g_p)=\deg\left(  (I(\Gamma_p)+\left\langle g_p\right\rangle):M_{p}^{\infty}\right)$
\STATE compute $\delta(\Gamma)$ by applying Remark \ref{rem:comp-inv-plane-curves} 
\IF {$\widetilde{d}(g_{p})=\deg(\Gamma)\cdot m-\deg \left\langle
G(p)\right\rangle-\delta(\Gamma)$}
\RETURN $\left\langle G\right\rangle $
\ENDIF
\ENDIF
\ENDIF
\ENDLOOP
\end{algorithmic}
\end{algorithm}

\begin{remark}
In Algorithm \ref{algModNormal}, the different $G(p)$ can be computed in
parallel. The individual computations can be parallelized by partitioning
the singular loci.
\end{remark}

\begin{remark}
The most expensive step of the verification is the computation of
$\delta(\Gamma)$. If we skip the verification, the algorithm will become
probabilistic, that is, the output is the Gorenstein adjoint ideal only with
high probability.
This usually accelerates the algorithm considerably and gives us, in
particular, a fast probabilistic way to compute both the geometric genus
$p(\Gamma)$ and $\deg\Delta(\mathfrak{G})=\dim_{\mathbb{Q}}\left(
\overline{\mathbb{Q}[C]}/\mathcal{C}_{\mathbb{Q}[C]}\right)  $.
\end{remark}

\section{Timings\label{Sec examples and comparisons}}

The algorithms for adjoint ideals presented in this paper are implemented in
the \textsc{{Singular}} library {\texttt{adjointideal.lib}} (see
\cite{adjideal}). They make use of the normalization algorithm of Section
\ref{sec normalization} either in its local or local to global variant, as
appropriate. These variants, in turn, are part of the \textsc{{Singular}}
library {\texttt{locnormal.lib}} (see \cite{BDLS2}).

In this section, we compare the performance of the different
algorithms. Specifically, we consider%

\begin{tabular}
[c]{ll}%
\texttt{LA} & Mnuk's global linear algebra approach (Algorithm
\ref{alg conductor linear algebra}),\\
\texttt{IQ} & the global ideal quotient approach (Algorithm
\ref{Alg ideal quotient gorenstein adjoint ideal}),\\
\texttt{locIQ} & the local ideal quotient approach (Algorithm
\ref{alg global from local} using Algorithm \ref{alg local integral quotient}%
),\\
\texttt{\texttt{locIQP2}} & the local ideal quotient approach for plane curves
with the\\
& improvements of Section \ref{sec summary} concerning ordinary multiple\\
& points and singularities of type $ADE$, and\\
\texttt{modLocIQ} & the modular local ideal quotient strategy (Algorithm
\ref{algModNormal}).
\end{tabular}

For the modular approach, we do not make use of a local analysis of the
singular locus except for computing the invariants needed in the verification step.

To quantify the improvement in computation time obtained by omitting the
verification step in the modular approach, we give timings for the
\emph{resulting, now probabilistic, version of Algorithm \ref{algModNormal}
(denoted by} \texttt{modLocIQ'} \emph{in the tables)}. In all examples
computed so far, the result of the probabilistic algorithm is indeed correct.

To quantify the contributions of the different normalization algorithms and to
provide a lower bound for any adjoint ideal algorithm using them, we also
specify the following computation times: normalization in \textsc{{Singular}}
via the local to global approach outlined in Section \ref{sec normalization}
(denoted by \texttt{locNormal}); and finding an integral basis in
\textsc{{Maple}} via the algorithm of van Hoeij (denoted by \texttt{Maple-IB}%
). Once being fully implemented in \textsc{{Singular}}, we expect further
improvements of the performance by computing the local contribution or just an
integral basis of the local ring by the algorithm discussed in \cite{BDLS}.
Since this algorithm and van Hoeij's algorithm rely on Puiseux series, they
work in characteristic zero only.

All timings are in seconds on an AMD Opteron $6174$ machine with $48$ cores,
$2.2$GHz, and $128$GB of RAM running a Linux operating system. A dash
indicates that the computation did not finish within $10000$ seconds. The
\emph{timings for parallel computations are marked by the symbol * and the
maximum number of cores used in parallel is indicated in brackets}.

\begin{remark}
All examples are defined over the field of rationals. For {\texttt{locIQ}%
$^{\mathtt{\ast}}$}, the number of cores used corresponds to the number of
components of the decomposition of the singular locus over $\mathbb{Q}$. For
{\texttt{modLocIQ}$^{\mathtt{\ast}}$}, the number of cores used in a given
iteration of the algorithm is obtained by summing up the number of components
modulo $p$ over all primes $p\in\mathcal{Q}$ chosen in Step 4 of Algorithm 
\ref{algModNormal}.
\end{remark}

To show the power of the modular algorithm, we give simulated parallel timings
even if the number of processes exceeds the number of cores available on our
machine (which is a valid approach since the algorithm has basically zero communication
overhead). For the single-core timings of \texttt{modLocIQ}, we \emph{indicate
in square brackets the number of primes used by the algorithm}.

Now we turn to explicit examples. First we consider rational plane curves
defined by a random parametrization of degree $n$. These curves have
$\binom{n-1}{2}$ ordinary double points. Their defining equations $f_{1,n}$
were generated by the function \texttt{randomRatCurve} from the {\sc{Singular}}
library {\texttt{paraplanecurves.lib}} (see \cite{BDLS3}), using the random seed 
$1$ and a random parametrization with coefficients of bitlength $15$.%
\[%
\scalebox{1}{
\begin{tabular}
[c]{l|rc|rc|rc|}
& \text{$f_{1,5}$} &  & \text{$f_{1,6}$} &  & \text{$f_{1,7}$} & \\\hline
{\small {$\deg$}} & 5 &  & 6 &  & 7 & \\\hline
{\small {\texttt{locNormal}}} & 2.1 &  & 56 &  & - & \\
{\small {\texttt{Maple-IB}}} & 5.1 &  & 47 &  & 318 & \\\hline
{\small {\texttt{LA}}} & 98 &  & 4400 &  & - & \\
{\small {\texttt{IQ}}} & 2.1 &  & 56 &  & - & \\
{\small {\texttt{locIQ}}} & 1.3 &  & 54 &  & 3800 & \\
{\small {\texttt{locIQ}$^{\mathtt{\ast}}$}} & 1.3 & \hspace{-0.1in}(1) & 54 &
\hspace{-0.25in}(1) & 3800 & \hspace{-0.25in}(1)\\
{\small {\texttt{locIQP2}}} & .18 &  & 1.2 &  & 49 & \\
{\small {\texttt{locIQP2}$^{\mathtt{\ast}}$}} & .18 & \hspace{-0.1in}(1) &
1.2 & \hspace{-0.25in}(1) & 49 & \hspace{-0.25in}(1)\\
{\small {\texttt{modLocIQ}}} & 6.4 & \hspace{-0.1in}[33] & 19 & \hspace
{-0.25in}[53] & 150 & \hspace{-0.25in}[75]\\
{\small {\texttt{modLocIQ'}}} & 6.2 & \hspace{-0.1in}[33] & 18 &
\hspace{-0.25in}[53] & 104 & \hspace{-0.25in}[75]\\
{\small {\texttt{modLocIQ}$^{\mathtt{\ast}}$}} & .36 & \hspace{-0.1in}(74) &
1.6 & \hspace{-0.15in}(153) & 51 & \hspace{-0.25in}(230)\\
{\small {\texttt{modLocIQ'}$^{\mathtt{\ast}}$}} & .21 & \hspace{-0.1in}(74) &
0.48 & \hspace{-0.15in}(153) & 5.2 & \hspace{-0.25in}(230)
\end{tabular}
}
\]
We observe that the detection of special types of singularities is fast and
yields the best performance among the non-probabilistic algorithms.

To compare the algorithms at a single singularity, we consider plane curves
with exactly one $A_{n}$ respectively $D_{n}$ singularity at the origin of the
affine chart $\{Z\neq0\}$ (ignoring singularities at infinity). For the
modular approach, we omit verification since this step relies on global
properties of the curve.

The curves with affine equation $f_{2,n,d}=Y^{2}+X^{n+1}+Y^{d}$, $n\geq1$,
$d\geq3$, have precisely one singularity of type $A_{n}$ at the origin:

{\footnotesize
\[%
\begin{tabular}
[c]{l|cc|cc|cc|cc|cc|cc|}
& \text{$f_{2,5,10}$} &  & \multicolumn{1}{c}{\text{$f_{2,5,100}$}} &  &
\text{$f_{2,5,500}$} &  & \text{$f_{2,50,100}$} &  & \text{$f_{2,50,500}$} &
& \text{$f_{2,400,500}$} & \\\hline
$\deg$ & \multicolumn{1}{r}{10} &  & \multicolumn{1}{r}{100} &  &
\multicolumn{1}{r}{\thinspace500} &  & \multicolumn{1}{r}{100} &  &
\multicolumn{1}{r}{500} &  & \multicolumn{1}{r}{500} & \\\hline
\texttt{locNormal} & \multicolumn{1}{r}{.12} &  & \multicolumn{1}{r}{.12} &  &
\multicolumn{1}{r}{.12} &  & \multicolumn{1}{r}{.51} &  &
\multicolumn{1}{r}{.51} &  & \multicolumn{1}{r}{3.6} & \\
\texttt{Maple-IB} & \multicolumn{1}{r}{.08} &  & \multicolumn{1}{r}{1.5} &  &
\multicolumn{1}{r}{96} &  & \multicolumn{1}{r}{4.7} &  &
\multicolumn{1}{r}{150} &  & \multicolumn{1}{r}{630} & \\\hline
\texttt{LA} & \multicolumn{1}{r}{.18} &  & \multicolumn{1}{r}{140} &  &
\multicolumn{1}{r}{-} &  & \multicolumn{1}{r}{150} &  & \multicolumn{1}{r}{-}
&  & \multicolumn{1}{r}{-} & \\
\texttt{IQ} & \multicolumn{1}{r}{.12} &  & \multicolumn{1}{r}{.12} &  &
\multicolumn{1}{r}{.12} &  & \multicolumn{1}{r}{.51} &  &
\multicolumn{1}{r}{.51} &  & \multicolumn{1}{r}{3.6} & \\
\texttt{modLocIQ'} & \multicolumn{1}{r}{.20} & \hspace{-0.1in}[2] &
\multicolumn{1}{r}{.22} & \hspace{-0.1in}[2] & \multicolumn{1}{r}{.96} &
\hspace{-0.15in}[2] & \multicolumn{1}{r}{1.1} & \hspace{-0.15in}[2] &
\multicolumn{1}{r}{2.0} & \hspace{-0.2in}[2] & \multicolumn{1}{r}{11} &
\hspace{-0.3in}[2]\\
\texttt{modLocIQ'}$^{\mathtt{\ast}}$ & \multicolumn{1}{r}{.10} &
\hspace{-0.1in}(2) & \multicolumn{1}{r}{.13} & \hspace{-0.1in}(2) &
\multicolumn{1}{r}{.48} & \hspace{-0.15in}(2) & \multicolumn{1}{r}{.54} &
\hspace{-0.15in}(2) & \multicolumn{1}{r}{1.2} & \hspace{-0.2in}(2) &
\multicolumn{1}{r}{5.8} & \hspace{-0.3in}(2)
\end{tabular}
\]
}

The curves with affine equation $f_{3,n,d}=X(X^{n-1}+Y^{2})+Y^{d}$, $n\geq2$,
$d\geq3$, have exactly one singularity of type $D_{n}$ at the
origin:{\footnotesize
\[%
\begin{tabular}
[c]{l|cc|cc|cc|cc|cc|cc|}
& \text{$f_{3,5,10}$} &  & \text{$f_{3,5,100}$} &  & \text{$f_{3,5,500}$} &  &
\text{$f_{3,50,100}$} &  & \text{$f_{3,50,500}$} &  & \text{$f_{3,400,500}$} &
\\\hline
$\deg$ & \multicolumn{1}{r}{10} &  & \multicolumn{1}{r}{100} &  &
\multicolumn{1}{r}{500} &  & \multicolumn{1}{r}{100} &  &
\multicolumn{1}{r}{50} &  & \multicolumn{1}{r}{500} & \\\hline
\texttt{locNormal} & \multicolumn{1}{r}{.15} &  & \multicolumn{1}{r}{.15} &  &
\multicolumn{1}{r}{.15} &  & \multicolumn{1}{r}{.67} &  &
\multicolumn{1}{r}{.67} &  & \multicolumn{1}{r}{4.9} & \\
\texttt{Maple-IB} & \multicolumn{1}{r}{.05} &  & \multicolumn{1}{r}{1.7} &  &
\multicolumn{1}{r}{100} &  & \multicolumn{1}{r}{34} &  &
\multicolumn{1}{r}{1830} &  & \multicolumn{1}{r}{-} & \\\hline
\texttt{LA} & \multicolumn{1}{r}{.20} &  & \multicolumn{1}{r}{140} &  &
\multicolumn{1}{r}{-} &  & \multicolumn{1}{r}{140} &  & \multicolumn{1}{r}{-}
&  & \multicolumn{1}{r}{-} & \\
\texttt{IQ} & \multicolumn{1}{r}{.15} &  & \multicolumn{1}{r}{.15} &  &
\multicolumn{1}{r}{.15} &  & \multicolumn{1}{r}{.67} &  &
\multicolumn{1}{r}{.67} &  & \multicolumn{1}{r}{5.0} & \\
\texttt{modLocIQ'} & \multicolumn{1}{r}{.22} & \hspace{-0.1in}[2] &
\multicolumn{1}{r}{.23} & \hspace{-0.1in}[2] & \multicolumn{1}{r}{.23} &
\hspace{-0.1in}[2] & \multicolumn{1}{r}{1.5} & \hspace{-0.15in}[2] &
\multicolumn{1}{r}{1.5} & \hspace{-0.2in}[2] & \multicolumn{1}{r}{24} &
\hspace{-0.35in}[2]\\
\texttt{modLocIQ'}$^{\mathtt{\ast}}$ & \multicolumn{1}{r}{.09} &
\hspace{-0.1in}(2) & \multicolumn{1}{r}{.10} & \hspace{-0.1in}(2) &
\multicolumn{1}{r}{.10} & \hspace{-0.1in}(2) & \multicolumn{1}{r}{.74} &
\hspace{-0.15in}(2) & \multicolumn{1}{r}{.77} & \hspace{-0.2in}(2) &
\multicolumn{1}{r}{17} & \hspace{-0.35in}(2)
\end{tabular}
\
\]
}

In both examples, the best strategy is \texttt{IQ} since we consider only one
singularity and since no coefficients of large bitlength occur.

The plane curves with defining equations%
\[
f_{4,n}=\left(  X^{n+1}+Y^{n+1}+Z^{n+1}\right)  ^{2}-4\left(  X^{n+1}%
Y^{n+1}+Y^{n+1}Z^{n+1}+Z^{n+1}X^{n+1}\right)
\]
were given in \cite{H} and have $3\left(  n+1\right)  $ singularities of type
$A_{n}$ if $n$ is even. To ensure that all singularities of the curves are in
the affine chart $\{Z\neq0\}$, we substitute $Z=2X-3Y+1$.
\[%
\begin{tabular}
[c]{l|cc|cc|cc|}
& \text{$f_{4,4}$} &  & \text{$f_{4,6}$} &  & \text{$f_{4,8}$} & \\\hline
{\small {$\deg$}} & \multicolumn{1}{r}{10} &  & \multicolumn{1}{r}{14} &  &
\multicolumn{1}{r}{18} & \\\hline
{\small {\texttt{locNormal}}} & \multicolumn{1}{r}{1.6} &  &
\multicolumn{1}{r}{-} &  & \multicolumn{1}{r}{-} & \\
{\small {\texttt{Maple-IB}}} & \multicolumn{1}{r}{2.2} &  &
\multicolumn{1}{r}{14} &  & \multicolumn{1}{r}{70} & \\\hline
{\small {\texttt{LA}}} & \multicolumn{1}{r}{89} &  & \multicolumn{1}{r}{-} &
& \multicolumn{1}{r}{-} & \\
{\small {\texttt{IQ}}} & \multicolumn{1}{r}{2.5} &  & \multicolumn{1}{r}{-} &
& \multicolumn{1}{r}{-} & \\
{\small {\texttt{locIQ}}} & \multicolumn{1}{r}{.96} &  & \multicolumn{1}{r}{-}
&  & \multicolumn{1}{r}{-} & \\
{\small {\texttt{locIQ}$^{\mathtt{\ast}}$}} & \multicolumn{1}{r}{.36} &
\hspace{-0.08in}(6) & \multicolumn{1}{r}{-} &  & \multicolumn{1}{r}{-} & \\
{\small {\texttt{locIQP2}}} & \multicolumn{1}{r}{1.0} &  &
\multicolumn{1}{r}{-} &  & \multicolumn{1}{r}{-} & \\
{\small {\texttt{locIQP2}$^{\mathtt{\ast}}$}} & \multicolumn{1}{r}{.38} &
\hspace{-0.08in}(6) & \multicolumn{1}{r}{-} &  & \multicolumn{1}{r}{-} & \\
{\small {\texttt{modLocIQ}}} & \multicolumn{1}{r}{3.7} & \hspace{-0.08in}[3] &
\multicolumn{1}{r}{23} & \hspace{-0.1in}[4] & \multicolumn{1}{r}{190} &
\hspace{-0.3in}[4]\\
{\small {\texttt{modLocIQ'}}} & \multicolumn{1}{r}{3.3} & \hspace
{-0.08in}[3] & \multicolumn{1}{r}{20} & \hspace{-0.1in}[4] &
\multicolumn{1}{r}{170} & \hspace{-0.3in}[4]\\
{\small {\texttt{modLocIQ}$^{\mathtt{\ast}}$}} & \multicolumn{1}{r}{.63} &
\hspace{-0.08in}(27) & \multicolumn{1}{r}{4.4} & \hspace{-0.1in}(48) &
\multicolumn{1}{r}{50} & \hspace{-0.3in}(48)\\
{\small {\texttt{modLocIQ'}$^{\mathtt{\ast}}$}} & \multicolumn{1}{r}{.38} &
\hspace{-0.08in}(27) & \multicolumn{1}{r}{2.2} & \hspace{-0.1in}(48) &
\multicolumn{1}{r}{30} & \hspace{-0.3in}(48)
\end{tabular}
\]

To conclude this section, we present examples of curves in higher-dimensional
projective space. As above, we first consider curves with only one singularity
in a given affine chart: let $L_{n}$ be the ideal of the image of%
\[
\mathbb{A}^{1}\longrightarrow\mathbb{A}^{3}\text{, }t\mapsto(t^{n-2},\text{
}t^{n-1},\text{ }t^{n})\text{.}%
\]
Second, denote by $I_{n}$ the ideal of the image in $\mathbb{P}^{5}$ under the
degree-$2$ Veronese embedding of the curve $\{f_{4,n}=0\}$. The resulting
timings are:%
\[%
\begin{tabular}
[c]{l|cc|cc|cc|cc|}
& \text{$L_{25}$} &  & \text{$L_{50}$} &  & \text{$I_{4}$} &  & \text{$I_{6}$}
& \\\hline
{\small {$\deg$}} & \multicolumn{1}{r}{25} &  & \multicolumn{1}{r}{50} &  &
\multicolumn{1}{r}{20} &  & \multicolumn{1}{r}{28} & \\\hline
{\small {\texttt{locNormal}}} & \multicolumn{1}{r}{3.9} &  &
\multicolumn{1}{r}{84} &  & \multicolumn{1}{r}{21} &  & \multicolumn{1}{r}{-}
& \\\hline
{\small {\texttt{IQ}}} & \multicolumn{1}{r}{3.9} &  & \multicolumn{1}{r}{84} &
& \multicolumn{1}{r}{30} &  & \multicolumn{1}{r}{-} & \\
{\small {\texttt{locIQ}}} & \multicolumn{1}{r}{3.9} &  &
\multicolumn{1}{r}{84} &  & \multicolumn{1}{r}{18} &  & \multicolumn{1}{r}{-}
& \\
{\small {\texttt{locIQ}$^{\mathtt{\ast}}$}} & \multicolumn{1}{r}{3.9} &
\hspace{-0.1in}(1) & \multicolumn{1}{r}{84} & \hspace{-0.2in}(1) &
\multicolumn{1}{r}{7.5} & \hspace{-0.1in}(6) & \multicolumn{1}{r}{-} & \\
{\small {\texttt{modLocIQ'}}} & \multicolumn{1}{r}{6.5} & \hspace{-0.1in}[2] &
\multicolumn{1}{r}{220} & \hspace{-0.2in}[2] & \multicolumn{1}{r}{74} &
\hspace{-0.1in}[5] & \multicolumn{1}{r}{2600} & \hspace{-0.4in}[5]\\
{\small {\texttt{modLocIQ'}$^{\mathtt{\ast}}$}} & \multicolumn{1}{r}{3.3} &
\hspace{-0.1in}(2) & \multicolumn{1}{r}{140} & \hspace{-0.2in}(2) &
\multicolumn{1}{r}{4.0} & \hspace{-0.1in}(45) & \multicolumn{1}{r}{59} &
\hspace{-0.4in}(69)
\end{tabular}
\]

To summarize, we observe that the ideal quotient approach is faster than the
linear algebra one. To some extent, this is due to the lack of efficiency of
the rational function arithmetic in \textsc{{Singular}}. The local strategy is
faster than the global one if there is more than one component in the
decomposition of the singular locus over $\mathbb{Q}$. In addition, the local
algorithm can be run in parallel and is, then, even faster. In most examples,
especially when the coefficients have large bitlength, the fastest approach is
the modular local strategy, which parallelizes in a two-fold way, by
localization and modularization. In contrast to other modular algorithms (such
as modular normalization), the verification step is usually very fast.

\vspace{0.1in} \noindent\emph{Acknowledgements}. We would like to thank
Christoph Lossen, Thomas Markwig, Mathias Schulze, and Frank Seelisch for
helpful discussions.


\end{document}